\documentclass{amsart}

\usepackage[colorlinks]{hyperref}
\usepackage{mathtools}
\mathtoolsset{showonlyrefs}
\usepackage[foot]{amsaddr}
\usepackage{bm}

\usepackage{natbib}
\setcitestyle{numbers,open={[},close={]}}

\theoremstyle{plain}
\newtheorem{theorem}{Theorem}[section]
\newtheorem{lemma}[theorem]{Lemma}
\newtheorem{corollary}[theorem]{Corollary}

\theoremstyle{definition}
\newtheorem{definition}[theorem]{Definition}
\newtheorem{assumption}[theorem]{Assumption}
\theoremstyle{remark}
\newtheorem{remark}[theorem]{Remark}

\DeclareMathOperator{\bary}{bar}
\DeclareMathOperator{\Bary}{Bar}
\DeclareMathOperator{\supp}{supp}
\DeclareMathOperator{\E}{\mathbb{E}}

\newcommand*{\rset}{\mathbb{R}}
\newcommand*{\nset}{\mathbb{N}}
\newcommand*{\qset}{\mathbb{Q}}
\newcommand*{\Prsp}[1]{\ensuremath{\mathcal{P}(#1)}}
\newcommand*{\diff}{\,d}
\newcommand*{\eqset}{\coloneqq}

\begin{document}
	\title{Fr{\'e}chet barycenters in the Monge--Kantorovich spaces}
	\thanks{$^*$\,The research was performed at IITP RAS with support of a Russian Science Foundation grant, project no. 14-50-00150.}
	\author{Alexey Kroshnin$^{1,2}$}
	\address{$^1$\,Moscow Institute of Physics and Technology}
	\address{$^2$\,Institute for Information Transmission Problems of RAS (Kharkevich Institute)}
	\email{kroshnin@phystech.edu}
	\date{2017}
	
	\begin{abstract}
		We consider the space \Prsp{X} of probability measures on arbitrary Radon space $X$ endowed with a transportation cost $J(\mu, \nu)$ generated by a nonnegative continuous cost function.
		For a probability distribution on~\Prsp{X} we formulate a notion of average with respect to this transportation cost, called here the \emph{Fr{\'e}chet barycenter}, prove a version of the law of large numbers for Fr{\'e}chet barycenters, and discuss the structure of~\Prsp{X} related to the transportation cost~$J$.
	\end{abstract}
	
	\maketitle

	\section{Introduction}\label{sec:intro}
	
	In this paper we consider averaging in the space~$\Prsp{X}$ of probability measures over a Radon space~$X$, using a transport optimization procedure to define a suitable concept of a ``typical element'', which extends the notion of \emph{Fr{\'e}chet mean}. For the first time a construction of this kind was introduced by M.~Agueh and G.~Carlier in \cite{AguehCarlier2011}: a \emph{Wasserstein barycenter} of a family of measures on the Euclidean space~$\rset^d$ is defined as the Fr{\'e}chet mean using the $2$-Wasserstein distance $W_2$ on~$\mathcal{P}(\rset^d)$, which is given by minimization of the mean-square displacement. In~\cite{AguehCarlier2011}, the authors establish existence, uniqueness, and regularity results for the Wasserstein barycenter and, when $d = 1$, provide an explicit formula for the Wasserstein barycenter in terms of quantile functions of the measures involved.
	
	Here we take a general Radon space~$X$ (e.g.\ Polish space) and a general transportation cost
	\begin{equation}
		J(\mu, \nu) = \inf \Bigl\{\int c(x, y) \diff\gamma(x, y) : \gamma \in \Prsp{X \times X},\, \pi^x_\# \gamma = \mu,\, \pi^y_\# \gamma = \nu\Bigr\} \ge 0,
	\end{equation}
	where~$c(\cdot, \cdot) \ge 0$ is a continuous \emph{cost function} that satisfies $c(x, y) = 0$ iff $x = y$. 
	
	Since $J(\mu, \nu) = 0$ iff $\mu = \nu$, this cost quantifies separation between measures $\mu$ and $\nu$ in $\Prsp{X}$ but does not necessarily satisfy the triangle inequality. Although $J(\cdot, \cdot)$ is not a metric, it generates a \emph{transportation topology} on $X$, and the space $X$ endowed with this topology is divided into equivalence classes each of which is a Radon space (in particular, the classes are separable and metrizable).
	
	Let a measure $\nu$ be fixed and $\bm\mu$ be a random element of~$\Prsp{X}$ with distribution $P_\mu$. We introduce a notion of Fr{\'e}chet typical element of~$P_\mu$ with respect to~$J(\cdot, \cdot)$, which we propose to call the \emph{Fr{\'e}chet barycenter} of~$P_\mu$. It is defined as any measure~$\nu$ for which the expected cost
	\begin{equation}
		\E J(\bm\mu, \nu) = \int_{\Prsp{X}} J(\mu, \nu) \diff P_\mu
	\end{equation}
	attains its minimum over~$\Prsp{X}$. Rigorous definitions of such a distribution and an integral are formulated in Section~\ref{sec:barycenters}. Suppose $\E J(\bm\mu, \cdot)$ is not identically equal to~$+\infty$ on~$\Prsp{X}$. Then there exists a Fr{\'e}chet barycenter of $P_\mu$. This averaging appears to be quite reasonable. Namely, if distributions $P_n$ converge to $P$ with respect to the transportation cost in $\mathcal{P}\bigl(\Prsp{X}\bigr)$ corresponding to $J(\cdot, \cdot)$ as cost function, then barycenters of $P_n$ also converge in some sense to the barycenter of $P$. For instance, this result implies a law of large numbers for Fr{\'e}chet barycenters.
	
	The paper is organized as follows.
	In Section~\ref{sec:MK_dist} we introduce some standard definitions and notations and recall properties of the Monge--Kantorovich distance.
	In Section~\ref{sec:topology} properties of the topology on~$\Prsp{X}$ induced by the Monge--Kantorovich transportation cost $J(\cdot, \cdot)$ are considered.
	In Section~\ref{sec:euclidean} we consider the particular case of $X = \rset^d$ and $c(x, y) = g(x - y)$, where $g(\cdot) \ge 0$ is a convex function.
	Then we define in Section~\ref{sec:barycenters} a generalized barycenter of distribution on $\Prsp{X}$ for general $X$. The central result of this paper is proved in Subsection~\ref{subsec:consistency}: the convergence of barycenters of distributions~$P_n$ is established provided $P_n$ themselves converge to some distribution~$P$.

	\section{The Monge--Kantorovich distance}\label{sec:MK_dist}

	\subsection{Notations}\label{subsec:notation}
	
	For a measurable space $X$ denote the space of probability measures on~$X$ by \Prsp{X}. In particular, if the space $X$ is topological, we assume it is endowed with the standard Borel $\sigma$-algebra $\mathcal{B}(X)$.
	 
	For two measurable spaces $X$, $Y$, a measurable map $T \colon X \to Y$ induces a map $T_\# \colon \Prsp{X} \to \Prsp{Y}$ given by $T_\# \mu(A) \eqset \mu\bigl(T^{-1}(A)\bigr)$ for any measurable $A \subset Y$. Recall that for any integrable function $f$
	\begin{equation}
		\int_Y f(y) \diff (T_\# \mu) = \int_X f\bigl(T(x)\bigr) \diff \mu.
	\end{equation}
	
	For a measure $\mu$ on $X$ and an integrable function $f \colon X \to \rset$ define the measure $f \lfloor \mu$ as 
	\[
	(f \lfloor \mu)(A) \eqset \int_A f(x) \diff \mu
	\] 
	for any measurable $A \subset Y$. Moreover, for any measurable set $B$ define $B \lfloor \mu \eqset \chi_B\lfloor \mu$, where $\chi_B(\cdot)$ is the characteristic function of~$B$.
	
	We will often drop the argument of function and the symbol of domain of integration if there is no risk of confusion.

	\subsection{The transport functional}\label{subsec:transp_distance}
	
	Let $(X, \rho)$ be a \emph{Radon space}, i.e.\ a separable metric space such that for every Borel probability measure $\mu \in \Prsp{X}$ and $\epsilon > 0$ there exists a compact set $D_\epsilon^\mu$ for which $\mu(X \setminus D_\epsilon^\mu) < \epsilon$ \cite[def.~5.1.4]{AmbrosioGigliSavare2008}. E.g., any Polish space is a Radon space.
	
	Fix a measurable nonnegative function $c \colon X \times X \to [0, \infty)$ and call it the \emph{cost function}. Assume $c(\cdot, \cdot)$ is continuous and \textit{consistent} in the sense that $c(x, x_n) \to 0$ iff $c(x_n, x) \to 0$ iff $x_n \to x$ for any $x \in X$, $\{x_n\}_{n \in \nset} \subset X$ (cf.\ discussion after Assumption~\ref{asm:loc_compact}). In particular, $c(x, y) = 0$ iff $x = y$.

	For two measures $\mu, \nu \in \Prsp{X}$ define the set of \emph{transport plans} taking $\mu$ to~$\nu$ as
	\begin{equation}
		\Pi(\mu, \nu) \eqset \bigl\{\gamma \in \Prsp{X \times X}: \pi^x_\# \gamma = \mu,\, \pi^y_\# \gamma = \nu\bigr\},
	\end{equation}
	where $\pi^x$ and~$\pi^y$ are the projections of~$X \times X$ to the first and second factor respectively. Observe that $\Pi(\mu, \nu)$ is always nonempty because it contains the direct product measure~$\mu \otimes \nu$. 

	The \emph{transportation cost} of a transport plan~$\gamma$ is defined as
	\begin{equation}
		K(\gamma) \eqset \int_{\rset \times \rset} c(x, y) \diff \gamma.
	\end{equation}
	It is easy to obtain that $K(\cdot)$ is lower semicontinuous with respect to the weak convergence of measures. The \emph{Monge--Kantorovich problem} for given $\mu, \nu \in \mathcal{P}(X)$ consists in minimizing the transportation cost $K(\gamma)$ over all $\gamma \in \Pi(\mu, \nu)$. Accordingly, the \emph{Monge--Kantorovich distance} (or \emph{transportation functional}) between $\mu$ and $\nu$ is the infimum of transportation costs:
	\begin{equation}
		J(\mu, \nu) \eqset \inf_{\gamma \in \Pi(\mu, \nu)} K(\gamma).
	\end{equation}
	 A transport plan $\gamma^*$ is called \emph{optimal} if $K(\cdot)$ attains its minimum over $\Pi(\mu, \nu)$ at~$\gamma^*$.
	
	\begin{theorem}\label{thm:optimal_plan}
		For any $\mu, \nu \in \Prsp{X}$ there exists an optimal transport plan.
	\end{theorem}
	
	\begin{proof}
		Let $J(\mu, \nu) < \infty$ and $\{\gamma_n\}_{n \in \nset} \subset \Pi(\mu, \nu)$ be a minimizing sequence for $K(\cdot)$. Observe that $\mu$ and $\nu$ are tight measures as $X$ is a Radon space, hence the set of transport plans $\Pi(\mu, \nu)$ is also tight. Indeed, for any $\epsilon > 0$ there exist such compact sets $D_\mu^\epsilon$, $D_\nu^\epsilon$ that $\mu(X \setminus D_\mu^\epsilon) < \epsilon / 2$ and $\nu(X \setminus D_\nu^\epsilon) < \epsilon / 2$; then $D^\epsilon \eqset D_\mu^\epsilon \times D_\nu^\epsilon \subset X \times X$ is compact and for any $\gamma \in \Pi(\mu, \nu)$ it holds
		\[
		\gamma(X \times X \setminus D^\epsilon) \le \gamma\bigl(X \times (X \setminus D_\nu^\epsilon)\bigr) + \gamma\bigl((X \setminus D_\mu^\epsilon) \times X\bigr) = \nu(X \setminus D_\nu^\epsilon) + \mu(X \setminus D_\mu^\epsilon) < \epsilon.
		\]
		Then by Prokhorov's theorem one can choose a weakly convergent subsequence $\gamma_{n_k} \rightharpoonup \gamma^* \in \Pi(\mu, \nu)$. Since $K(\cdot)$ is lower semicontinuous with respect to the  weak convergence,
		\[
		K(\gamma^*) \le \liminf K(\gamma_{n_k}) = \inf_{\gamma \in \Pi(\mu, \nu)} K(\gamma) \eqset J(\mu, \nu),
		\]
		i.e.\ $\gamma^*$ is an optimal transport plan from $\mu$ to $\nu$.
	\end{proof}
	
	Notice, however, that optimal transport plan from $\mu$ to $\nu$ may be not unique.
	
	\begin{corollary}
		$J(\mu, \nu) = 0$ iff $\mu = \nu$.
	\end{corollary}
	
	Let us recall some properties of the transportation functional $J(\cdot, \cdot)$ (see e.g.\ \cite{AmbrosioGigliSavare2008,Santambrogio2015,Villani2009}).
	
	\begin{lemma}\label{lem:convex}
		The functional $J(\cdot, \cdot)$ is convex, i.e.\ for any measures $\mu_0, \nu_0, \mu_1, \nu_1 \in \Prsp{X}$ and $t \in [0, 1]$ it holds that
		\[
		J(\mu_t, \nu_t) \le (1 - t) J(\mu_0, \nu_0) + t J(\mu_1, \nu_1),
		\]
		where $\mu_t \eqset (1 - t) \mu_0 + t \mu_1$, $\nu_t \eqset (1 - t) \nu_0 + t \nu_1$.
	\end{lemma}

	\begin{proof}
		Let $\gamma_s$ be an optimal transport plan from $\mu_s$ to $\nu_s$, $s \in \{0, 1\}$. Then $\gamma_t \eqset (1 - t) \gamma_0 + t \gamma_1 \in \Pi(\mu_t, \nu_t)$, hence
		\[
		J(\mu_t, \nu_t) \le K(\gamma_t) = (1 - t) K(\gamma_0) + t K(\gamma_1) = (1 - t) J(\mu_0, \nu_0) + t J(\mu_1, \nu_1).\qedhere
		\]
	\end{proof}

	\begin{corollary}\label{cor:lip}
		Fix $\mu_0, \mu_1 \in \Prsp{X}$. Then for any $0 \le t < t' \le 1$ 
		\[
		J(\mu_t, \mu_{t'}) \le (t' - t) J(\mu_0, \mu_1).
		\]
	\end{corollary}

	\begin{proof}
		From the convexity of $J(\cdot, \cdot)$ it follows that
		\begin{multline}
			J(\mu_t, \mu_{t'}) \le \frac{1 - t'}{1 - t} J(\mu_t, \mu_t) + \frac{t' - t}{1 - t} J(\mu_t, \mu_1)\\
			{} \le \frac{t' - t}{1 - t} \bigl((1 - t) J(\mu_0, \mu_1) + t J(\mu_1, \mu_1)\bigr) = (t' - t) J(\mu_0, \mu_1).\qedhere
		\end{multline}
	\end{proof}
	
	\begin{lemma}\label{lem:l.s.c.}
		The functional $J(\cdot, \cdot)$ is lower semicontinuous with respect to the weak convergence of measures.
	\end{lemma}
	
	\begin{proof}
		Let $\mu_n \rightharpoonup \mu$, $\nu_n \rightharpoonup \nu$ and $\gamma_n \in \Pi(\mu_n, \nu_n)$ be an optimal transport plan from $\mu_n$ to $\nu_n$. Notice that $\{\mu_n\}_{n \in \nset}$ and $\{\nu_n\}_{n \in \nset}$ weakly converge so they are tight \cite[p.~108]{AmbrosioGigliSavare2008} and $\{\gamma_n\}_{n \in \nset}$ is the same. Consider subsequence whose lower limit $\liminf J(\mu_n, \nu_n) \in [0, \infty]$ is attained. It contains a weakly convergent subsequence $\gamma_{n_k} \rightharpoonup \tilde\gamma \in \Pi(\mu, \nu)$. So, due to the lower semicontinuity of $K(\cdot)$ one can obtain 
		\begin{equation}
			J(\mu, \nu) \le K(\tilde\gamma) \le \liminf K(\gamma_{n_k}) = \lim J(\mu_{n_k}, \nu_{n_k}) = \liminf J(\mu_n, \nu_n).\qedhere
		\end{equation}
	\end{proof}
	
	\begin{corollary}
		$J(\cdot, \cdot)$ is measurable with respect to the product of Borel $\sigma$-algebras $\mathcal{B}_w\bigl(\Prsp{X}\bigr) \otimes \mathcal{B}_w\bigl(\Prsp{X}\bigr)$ induced by the topology of weak convergence.
	\end{corollary}
	
	\begin{theorem}\label{thm:strong_to_weak}
		Let $\{\nu_n\}_{n \in \nset}$ be a sequence of measures from \Prsp{X} such that\\
		${J(\nu^*, \nu_n) \to 0}$ \textup{(}or $J(\nu_n, \nu^*) \to 0$\textup{)} for some $\nu^* \in \Prsp{X}$; then $\nu_n \rightharpoonup \nu^*$.
	\end{theorem}
	
	\begin{proof}
		Assume $\{\nu_n\}_{n \in \nset}$ fails to converge to $\nu^*$. Then there exists a closed set $F \subset X$ such that $\limsup \nu_n(F) > \nu^*(F)$. Let $3 \epsilon \eqset \lim \nu_n(F) - \nu^*(F) > 0$ without relabelling. Consider the following open neighbourhood of $F$:
		\[
		F_r \eqset \bigl\{x \in X: \inf_{y \in F} c(x, y) < r\bigr\} \supset F,\; r > 0.
		\]
		For any $r > 0$ set $F_r$ is open due to continuity of $c(\cdot, \cdot)$. One can obtain that for any $x \notin F$ there exist an ``open ball'' $B_r^c(x) \eqset \bigl\{y \in X: c(x, y) < r\bigr\}$ such that $B_r^c(x) \cap F = \emptyset$. Therefore $\bigcap_{r > 0} F_r = F$ and $\nu^*(F) = \lim_{r \to 0} \nu^*(F_r)$. Let $r_0 > 0$ be such that $\nu^*(F_{r_0}) < \nu^*(F) + \epsilon$. As $\nu_n(F) > \nu^*(F) + 2 \epsilon$ starting from some $n$,
		\[
		\gamma_n\bigl((X \setminus F_{r_0}) \times F\bigr) = \nu_n(F) - \gamma_n(F_{r_0} \times F) \ge \nu_n(F) - \nu^*(F_{r_0}) > \epsilon,
		\]
		where $\gamma_n \in \Pi(\nu^*, \nu_n)$ is an optimal transport plan. Consequently,
		\[
		J(\nu^*, \nu_n) = K(\gamma_n) \ge \gamma_n\bigl((X \setminus F_{r_0}) \times F\bigr) \inf\bigl\{c(x, y): x \notin F_{r_0},\, y \in F\bigr\} \ge \epsilon r_0
		\]
		and $\liminf J(\nu^*, \nu_n) \ge \epsilon r_0 > 0$, which contradicts the condition of the theorem.
	\end{proof}
	
	As we have seen, convergence with respect to the transportation functional implies the weak convergence. Actually, the converse also holds under some additional assumptions.
	
	\begin{theorem}\label{thm:bound_weak_to_strong}
		Let $\nu_n \rightharpoonup \nu^*$ and $\supp \nu_n \subset F$ for all $n$, where $F \subset X$ is a closed set such that $\sup_{x, y \in F} c(x, y) < \infty$. Then $\lim J(\nu_n, \nu^*) = \lim J(\nu^*, \nu_n) = 0$.
	\end{theorem}
	
	\begin{proof}
		Fix $\epsilon > 0$. Due to separability of $X$ and continuity of $c(\cdot, \cdot)$ one can cover $X$ with a countable union of closed balls $\{\overline{B}_{r_i}(x_i)\}_{i \in \nset}$ such that $c(x, y) < \epsilon$ whenever $x, y \in \overline{B}_{2 r_i}(x_i)$ for all $i$. Fix $m \in \nset$ such that $\nu^*\bigl(X \setminus \bigcup_{i = 1}^m B_{r_i}(x_i)\bigr) < \epsilon$ and consider continuous functions $f_i(\cdot) \colon X \to [0, 1]$, $0 \le i \le m$ satisfying
		\begin{align}
			\sum_{i = 0}^m f_i(x) &\equiv 1;\\
			f_0(x) &= 0,\; x \in \overline{B}_{r_i}(x_i),\; 1 \le i \le m;\\
			f_i(x) &= 0,\; x \notin \overline{B}_{2 r_i}(x_i),\; 1 \le i \le m.
		\end{align}
		Without loss of generality $\int f_i \diff\nu^* > 0$ for all $i$. Define measures
		\begin{align}
			\lambda_n^i &\eqset \frac{(f_i\lfloor \nu_n) \otimes (f_i\lfloor \nu^*)}{\max\{\int f_i \diff\nu_n, \int f_i \diff\nu^*\}},\; 1 \le i \le m,\; n \in \nset;\\
			\hat\nu_n &\eqset \nu_n - \sum_{i = 1}^m \pi^x_\# \lambda_n^i = \biggl(1 - \sum_{i = 1}^m \frac{\int f_i \diff\nu^*}{\max\{\int f_i \diff\nu_n, \int f_i \diff\nu^*\}} f_i\biggr)\lfloor \nu_n\ge f_0\lfloor \nu_n \ge 0;\\
			\hat\nu_n^* &\eqset \nu^* - \sum_{i = 1}^m \pi^y_\# \lambda_n^i = \biggl(1 - \sum_{i = 1}^m \frac{\int f_i \diff\nu_n}{\max\{\int f_i \diff\nu_n, \int f_i \diff\nu^*\}} f_i\biggr)\lfloor \nu^*\ge f_0\lfloor \nu^* \ge 0.
		\end{align}
		Since $\nu_n \rightharpoonup \nu^*$
		\begin{multline}
			\hat\nu_n(X) = \hat\nu_n^*(X) = \int \biggl(1 - \sum_{i = 1}^m \frac{\int f_i \diff\nu_n}{\max\{\int f_i \diff\nu_n, \int f_i \diff\nu^*\}} f_i\biggr) \diff\nu^*\\
			{} \to \int \Bigl(1 - \sum_{i = 1}^m f_i\Bigr) \diff\nu^* = \int f_0 \diff\nu^*\le \nu^*\Bigl(X \setminus \bigcup_{i = 1}^m B_{r_i}(x_i)\Bigr) < \epsilon.
		\end{multline}
		
		Consider the following transport plans:
		\[
		\gamma_n \eqset \frac{\hat\nu_n \otimes \hat\nu_n^*}{\hat\nu_n(X)} + \sum_{i = 1}^m \lambda_n^i \in \Pi(\nu_n, \nu^*),\; n \in \nset.
		\]
		From $\supp \nu_n \subset F$ it follows $\supp \nu^* \subset F$ and $\supp \gamma_n \subset F \times F$. Define $M \eqset {\sup_{x, y \in F} c(x, y) < \infty}$. $\supp \lambda_n^i \subset \overline{B}_{2 r_i}(x_i) \times \overline{B}_{2 r_i}(x_i)$ by the definition of functions $f_i(\cdot)$. Now one can obtain that
		\begin{multline}
			\limsup J(\nu_n, \nu^*) \le \limsup K(\gamma_n) \le \limsup \biggl(M \frac{\hat\nu_n(X) \hat\nu_n^*(X)}{\hat\nu_n(X)} + \epsilon \sum_{i = 1}^m \lambda_n^i(X \times X)\biggr)\\
			{} \le M \limsup \hat\nu_n^*(X) + \epsilon \le (1 + M) \epsilon.
		\end{multline}
		It proves that $J(\nu_n, \nu^*) \to 0$ because of the arbitrary choice of $\epsilon$. In the same way one can show that $J(\nu^*, \nu_n) \to 0$.
	\end{proof}

	\section{Transportation topology}\label{sec:topology}
	
	Assume that the cost function $c(\cdot, \cdot)$ satisfies the following weak triangle inequalities.
	
	\begin{assumption}\label{asm:general_growth_condition}
		There exist constants $A, B \ge 0$ such that the following set of inequalities holds for all $x, y, z \in X$:
		\begin{align}
			c(x, y) &\le A + B \bigl(c(x, z) + c(y, z)\bigr),\\
			c(x, y) &\le A + B \bigl(c(x, z) + c(z, y)\bigr),\\
			c(x, y) &\le A + B \bigl(c(z, x) + c(y, z)\bigr),\\
			c(x, y) &\le A + B \bigl(c(z, x) + c(z, y)\bigr).
		\end{align}
	\end{assumption}
	
	This is a quite natural assumption which holds for a wide class of functions, e.g.\ for $c(x, y) = \rho^p(x, y)$, where $\rho(\cdot, \cdot)$ is a metric on $X$ and $p > 0$ (the case of $p$-Wasserstein spaces).
	
	Let us now show that the Monge--Kantorovich distance ``inherits'' the inequalities for the cost function.
	
	\begin{lemma}\label{lem:general_weak_triangle}
		For all $\mu, \nu, \lambda \in \Prsp{X}$
		\begin{align}
			J(\mu, \nu) &\le A + B \bigl(J(\mu, \lambda) + J(\nu, \lambda)\bigr),\\
			J(\mu, \nu) &\le A + B \bigl(J(\mu, \lambda) + J(\lambda, \nu)\bigr),\\
			J(\mu, \nu) &\le A + B \bigl(J(\lambda, \mu) + J(\nu, \lambda)\bigr),\\
			J(\mu, \nu) &\le A + B \bigl(J(\lambda, \mu) + J(\lambda, \nu)\bigr).
		\end{align}
	\end{lemma}
	
	\begin{proof}
		Consider measures $\mu, \nu, \lambda \in \Prsp{X}$ and optimal plans $\gamma_1 \in \Pi(\mu, \lambda)$, $\gamma_2 \in \Pi(\nu, \lambda)$. By disintegration theorem \cite[Theorem~5.3.1]{AmbrosioGigliSavare2008} there exists a measure $\sigma \in \Pi(\mu, \nu, \lambda)$ such that $\pi^{x, z}_\# \sigma = \gamma_1$ and $\pi^{y, z}_\# \sigma = \gamma_2$ \cite[Lemma~5.3.2]{AmbrosioGigliSavare2008}. Now, applying the weak triangle inequality, one can obtain
		\begin{multline}
			J(\mu, \nu) \le K(\pi^{x, y}_\# \sigma) = \int c(x, y) \diff \sigma \le \int \Bigl(A + B \bigl(c(x, z) + c(y, z)\bigr)\Bigr) \diff \sigma\\
			{} = A + B \bigl(K(\gamma_1) + K(\gamma_2)\bigr) = A + B \bigl(J(\mu, \lambda) + J(\nu, \lambda)\bigr).
		\end{multline}
		The other inequalities might be proved similarly.
	\end{proof}

	Notice that if $c(\cdot, \cdot)$ is a metric on $X$, then $J(\cdot, \cdot)$ is a metric on $\Prsp{X}$ (which may make the value $+\infty$). Moreover, it is an inner metric, even if $X$ is a disconnected space. Indeed, if $J(\mu, \nu) < \infty$, then the curve given by $[0, 1] \owns t \mapsto (1 - t) \mu + t \nu$ is a minimizing geodesic connecting $\mu$ to $\nu$ due to Corollary~\ref{cor:lip}.
	
	\begin{lemma}\label{lem:compact_growth_condition}
		Let $D \subset X$ be a compact set. Then for any $\epsilon > 0$ there exist an open neighbourhood $U_\epsilon(D) \supset D$ and a constant $B_\epsilon^D$ such that
		\begin{align}
			c(x, y) &\le \epsilon + c(x, z) + B_\epsilon^D c(y, z),\\
			c(x, y) &\le \epsilon + c(z, y) + B_\epsilon^D c(z, x)
		\end{align}
		for all $x, y, z \in U_\epsilon(D)$.
	\end{lemma}
	
	\begin{proof}
		Fix $\epsilon > 0$. As $c(\cdot, \cdot)$ is continuous, it is uniformly continuous on $D \times D$, hence there exists an open set $V \in X \times X$ such that $(y, y) \in V$ for all $y \in D$ and $c(x, y) < c(x, z) + \epsilon / 2$ for all $x \in X$, $(y, z) \in V$. Define $M \eqset \max_{x, y \in D} c(x, y) < \infty$ and 
		\[\delta \eqset \min_{(y, z) \in D^2 \setminus V} c(y, z) > 0\]
		which is positive due to compactness of $D$. If $(y, z) \in D^2 \setminus V$ then $c(x, y) \le M \le \frac{M}{\delta} c(y, z)$ for all $x$. Consequently,
		\begin{equation}\label{eq:compact_cond}
			c(x, y) \le \epsilon / 2 + c(x, z) + \frac{M}{\delta} c(y, z)
		\end{equation}
		for all $x, y, z \in D$. Denote $\frac{M}{\delta}$ by $B_\epsilon^D$.
		
		Due to continuity of $c(\cdot, \cdot)$ one can choose an open neighbourhood $W$ of $D^3$ such that for all $x, y, z \in W$ there exist $x', y', z' \in D$ for which $|c(x, y) - c(x', y')| < \gamma$, $|c(x, z) - c(x', z')| < \gamma$ and $|c(y, z) - c(y', z')| < \gamma$ where $\gamma \eqset \epsilon / \bigr(6 (1 + B_\epsilon^D)\bigr)$. As $D^3$ is compact, $\rho(D^3, X^3 \setminus W) > 0$ and there exists such neighbourhood $U_\epsilon(D)$ that $\bigl(U_\epsilon(D)\bigr)^3 \subset W$. Now, applying inequality~\eqref{eq:compact_cond} and the definion of $W$ one can obtain that
		\begin{multline}
			c(x, y) \le \gamma + c(x', y') \le \gamma + \epsilon / 2 + c(x', z') + B_\epsilon^D c(y', z')\\
			{} \le \gamma (2 + B_\epsilon^D) + \epsilon / 2 + c(x, z) + B_\epsilon^D c(y, z) \le \epsilon + c(x, z) + B_\epsilon^D c(y, z)
		\end{multline}
		for all $x, y, z \in U_\epsilon(D)$. The second inequality can be treated in the same way.
	\end{proof}
	
	\begin{lemma}[continuity]
		Take two sequences $\{\mu_n\}_{n \in \nset}$, $\{\nu_n\}_{n \in \nset}$ such that\\
		$J(\mu^*, \mu_n) \to 0$ and $J(\nu^*, \nu_n) \to 0$ for some measures $\mu^*, \nu^* \in \Prsp{X}$. Then $J(\mu_n, \nu_n) \to J(\mu^*, \nu^*)$.
	\end{lemma}
	
	\begin{proof}
		Let $\gamma_n^1 \in \Pi(\mu^*, \mu_n)$, $\gamma^2 \in \Pi(\mu^*, \nu^*)$, $\gamma_n^3 \in \Pi(\nu^*, \nu_n)$ be optimal transport plans. Consider measures $\sigma_n \in \Prsp{X^4}$ such that $(\pi^{x_2} \times \pi^{x_1})_\# \sigma_n = \gamma_n^1$,  $(\pi^{x_2} \times \pi^{x_3})_\# \sigma_n = \gamma^2$ and $(\pi^{x_3} \times \pi^{x_4})_\# \sigma_n = \gamma_n^3$. Since the sequences are tight one can fix $\epsilon > 0$ and a compact set $D$ such that $\mu_n(X \setminus D)$, $\mu^*(X \setminus D)$, $\nu_n(X \setminus D)$, $\nu^*(X \setminus D)$ and $\int_{X^2 \setminus D^2} c(x_2, x_3) \diff \gamma^2$ are less than $\epsilon$.
		Obviously,
		\[
		J(\mu_n, \nu_n) \le K\bigl((\pi^{x_1} \times \pi^{x_4})_\# \sigma_n\bigr) = \int c(x_1, x_4) \diff \sigma_n.
		\]
		
		Consider the set $Y \eqset U_\epsilon(D) \times D^2 \times U_\epsilon(D)$. Now one can obtain due to Lemma~\ref{lem:compact_growth_condition} that
		\begin{multline}
			\int_Y c(x_1, x_4) \diff \sigma_n \le \int_Y \bigl(\epsilon + B_\epsilon^D c(x_2, x_1) + (1 + \epsilon) c(x_2, x_4)\bigr) \diff \sigma_n\\
			{} \le \int_Y \bigl(\epsilon + B_\epsilon^D c(x_2, x_1) + (1 + \epsilon) \epsilon + (1 + \epsilon) B_\epsilon^D c(x_3, x_4) + (1 + \epsilon)^2 c(x_2, x_3)\bigr) \diff \sigma_n\\
			{} \le \epsilon + B_\epsilon^D K(\gamma_n^1) + (1 + \epsilon) \epsilon + (1 + \epsilon) B_\epsilon^D K(\gamma_n^3) + (1 + \epsilon)^2 K(\gamma^2)\\
			{} = \epsilon + B_\epsilon^D J(\mu^*, \mu_n) + (1 + \epsilon) \epsilon + (1 + \epsilon) B_\epsilon^D J(\nu^*, \nu_n) + (1 + \epsilon)^2 J(\mu^*, \nu^*)\\
			{} \xrightarrow[n \to \infty]{} \epsilon + (1 + \epsilon) \epsilon + (1 + \epsilon)^2 J(\mu^*, \nu^*) \to J(\mu^*, \nu^*) \text{ as } \epsilon \to 0.
		\end{multline}
		The remaining term may be bounded by Assumption~\ref{asm:general_growth_condition} in the following way:
		\begin{multline}
			\int_{X^4 \setminus Y} c(x_1, x_4) \diff \sigma_n \le \int_{X^4 \setminus Y} \bigl(A + B c(x_2, x_1) + B c(x_2, x_4)\bigr) \diff \sigma_n\\
			{} \le \int_{X^4 \setminus Y} \bigl(A + B c(x_2, x_1) + A B + B^2 c(x_2, x_3) + B^2 c(x_3, x_4)\bigr) \diff \sigma_n\\
			{} \le (A + A B) \sigma_n(X^4 \setminus Y) + B J(\mu^*, \mu_n) + B^2 \int_{X^4 \setminus Y} c(x_2, x_3) \diff \sigma_n + B^2 J(\nu^*, \nu_n)\\
			{} \le 4 (A + A B) \epsilon + B J(\mu^*, \mu_n) + B^2 J(\nu^*, \nu_n) + B^2 \int_{X^4 \setminus Y} c(x_2, x_3) \diff \sigma_n.
		\end{multline}
		Notice that $X^4 \setminus Y = \Bigl[X \times (X^2 \setminus D^2) \times X\Bigr] \cup \Bigl[\bigl(X \setminus U_\epsilon(D)\bigr) \times D^2 \times X\Bigr] \cup \Bigl[X \times D^2 \times \bigl(X \setminus U_\epsilon(D)\bigr)\Bigr]$.
		\[
		\int_{X \times (X^2 \setminus D^2) \times X} c(x_2, x_3) \diff \sigma_n = \int_{X^2 \setminus D^2} c(x_2, x_3) \diff \gamma^2 < \epsilon.
		\]
		Moreover, since $J(\mu^*, \mu_n) \to 0$ and $J(\nu^*, \nu_n) \to 0$, $\gamma_n^1 \rightharpoonup (x \times x)_\# \mu^*$ and $\gamma_n^3 \rightharpoonup (x \times x)_\# \nu^*$, so $\sigma_n \rightharpoonup (\pi^x \times \pi^x \times \pi^y \times \pi^y)_\# \gamma^2$. Since $\bigl(X \setminus U_\epsilon(D)\bigr) \times D^2 \times X$ is a closed set and $c(x_2, x_3)$ is continuous and bounded on it we have that
		\[
		\limsup \int_{\bigl(X \setminus U_\epsilon(D)\bigr) \times D^2 \times X} c(x_2, x_3) \diff \sigma_n \le \int_{\bigl(D \cap X \setminus U_\epsilon(D)\bigr) \times D} c(x, y) \diff \gamma^2 = 0,
		\]
		as $D \cap X \setminus U_\epsilon(D) = \emptyset$. In the same way one can obtain that 
		\[
		\limsup \int_{X \times D^2 \times \bigl(X \setminus U_\epsilon(D)\bigr)} c(x_2, x_3) \diff \sigma_n = 0.
		\] 
		Thus $J(\mu_n, \nu_n) \to J(\mu^*, \nu^*)$.
	\end{proof}
	
	\begin{corollary}[topology]\label{cor:topology}
		The set of all balls $B_r^J(\mu) \eqset \bigl\{\nu \in \Prsp{X}: J(\mu, \nu) < r\bigr\}$ form a basis of a \emph{topology $\tau_J$} on \Prsp{X}, and $J(\cdot, \cdot)$ is continuous with respect to this topology.
	\end{corollary}
	
	Let us denote convergence of a sequence $\{\nu_n\}_{n \in \nset}$ to $\nu^*$ in the topology $\tau_J$ (\emph{the transportation convergence}) by $\nu_n \xrightarrow{J} \nu^*$. It is equivalent to $J(\nu_n, \nu^*) \to 0$ and $J(\nu^*, \nu_n) \to 0$. If $X$ is a compact space, it follows from Theorems~\ref{thm:strong_to_weak} and~\ref{thm:bound_weak_to_strong} that weak convergence of measures is equivalent to the transportation convergence. But in this case the space \Prsp{X} with the topology of weak convergence is compact itself, and so is $\bigl(\Prsp{X}, \tau_J\bigr)$. Notice that if $X$ is not compact, \Prsp{X} is neither compact nor locally compact.
	
	\begin{lemma}\label{lem:assoc+symmetr}
		Let $\{\mu_n\}_{n \in \nset}$, $\{\nu_n\}_{n \in \nset}$, $\{\lambda_n\}_{n \in \nset}$ be tight sequences such that\\
		$J(\mu_n, \nu_n) \to 0$ and $J(\nu_n, \lambda_n) \to 0$; then $J(\nu_n, \mu_n) \to 0$ and $J(\mu_n, \lambda_n) \to 0$. 
	\end{lemma}
	
	\begin{proof}
		Fix $\epsilon > 0$ and a compact set $D \subset X$ such that $\mu_n(D)$, $\nu_n(D)$ and $\lambda_n(D)$ are greater than  $1 - \epsilon$. Consider again measures $\sigma_n$ such that $\pi^{x, y}_\# \sigma_n = \gamma_n^1$ and $\pi^{y, z}_\# \sigma_n = \gamma_n^2$ where $\gamma_n^1$ and $\gamma_n^2$ are optimal transport plans from $\mu_n$ to $\nu_n$ and from $\nu_n$ to $\lambda_n$, respectively. Due to Assumption~\ref{asm:general_growth_condition} and Lemma~\ref{lem:compact_growth_condition} one can obtain that
		\begin{multline}
			J(\mu_n, \lambda_n) \le K\bigl((\pi^x \times \pi^z)_\# \sigma_n\bigr) = \int c(x, z) \diff \sigma_n\\
			{} \le \int_{D^3} \bigl(\epsilon + (1 + \epsilon) c(x, y) + B_\epsilon^D c(y, z)\bigr) \diff \sigma_n + \int_{X^3 \setminus D^3} \bigl(A + B c(x, y) + B c(y, z)\bigr) \diff \sigma_n\\
			{} \le \epsilon + (1 + \epsilon) J(\mu_n, \nu_n) + B_\epsilon^D J(\nu_n, \lambda_n) + 3 \epsilon A + B J(\mu_n, \nu_n) + B J(\nu_n, \lambda_n)\\
			{} \to \epsilon + 3 \epsilon A \to 0 \text{ as } \epsilon \to 0.
		\end{multline}
		Thus $J(\mu_n, \lambda_n) \to 0$. Similarly,
		\begin{multline}
			J(\nu_n, \mu_n) \le K\bigl((\pi^y \times \pi^x)_\# \gamma_n^1\bigr) = \int c(y, x) \diff \gamma_n^1\\
			{} \le \int_{D^3} \bigl(\epsilon + B_\epsilon^D c(x, y)\bigr) \diff \gamma_n^1 + \int_{X^3 \setminus D^3} \bigl(A + B c(x, y)\bigr) \diff \gamma_n^1\\
			{} \le \epsilon + B_\epsilon^D J(\mu_n, \nu_n) + 2 \epsilon A + B J(\mu_n, \nu_n) \to \epsilon + 2 \epsilon A,
		\end{multline}
		therefore $J(\nu_n, \mu_n) \to 0$.
	\end{proof}
	
	Let the relation $\mu \sim \nu$ be defined as $J(\mu, \nu) < \infty$. Then it is an equivalence on~\Prsp{X} and splits the space into equivalence classes $E(\mu) \eqset \bigl\{\nu \in \Prsp{X}: J(\mu, \nu) < \infty\bigr\}$. Notice that every equivalence class is path-connected, even if $X$ is disconnected, since curve $[0, 1] \owns t \mapsto (1 - t) \mu + t \nu$ is continuous by Corollary~\ref{cor:lip}, whenever $J(\mu, \nu) < \infty$. Let us denote by $E_0$ the class containing delta-measures, i.e.\ $E_0 \eqset E(\delta_{x_0}) = \bigl\{\nu \in \Prsp{X}: \int c(x, x_0) \diff \nu(x) < \infty\bigr\}$ for an arbitrary $x_0 \in X$ (obviously, it does not depend on the choice of $x_0$). 
	
	Consider the following useful construction: fix some point $x_0 \in X$ and for given $R > 0$ take a continuous function $f_R \colon X \times X \to [0, 1]$ such that $f_R(x, y) = 1$ for $x, y \in B_R^c(x_0)$ and $f_R(x, y) = 0$ if $x \notin B_{R + 1}^c(x_0)$ or $y \notin B_{R + 1}^c(x_0)$. Let us take measures $\mu$, $\nu$, $\gamma \in \Pi(\mu, \nu)$ and consider $\lambda \eqset f_R \lfloor \gamma$. Define
	\[
	\tilde\gamma \eqset \gamma - \lambda + (\pi^y \times \pi^y)_\# \lambda
	\]
	and $\tilde\nu \eqset (\pi^x)_\# \tilde\gamma$. So $\tilde\gamma \in \Pi(\tilde\nu, \nu)$ and $J(\tilde\nu, \nu) \le K(\tilde\gamma) = K(\gamma) - K(\lambda) = K(\gamma) - K'(\gamma)$ where $K'(\gamma) = K'_R(\gamma) \eqset K(f_R \lfloor \gamma)$.
	
	Now consider a weakly convergent sequence of plans $\Pi(\mu, \nu_n) \owns \gamma_n \rightharpoonup \gamma^* \in \Pi(\mu, \nu^*)$. One has $\tilde\gamma_n \rightharpoonup \tilde\gamma^*$ hence $\tilde\nu_n \rightharpoonup \tilde\nu^*$. But on the complement of the ball $B_{R + 1}^c(x_0)$ all the measures $\tilde\nu$ coincide with $\mu$, consequently, $\tilde\nu_n \xrightarrow{J} \tilde\nu^*$ by Theorem~\ref{thm:bound_weak_to_strong}.
	
	Theorems~\ref{thm:strong_to_weak} and~\ref{thm:bound_weak_to_strong} state some results about the relation between the transportation convergence and the weak convergence. The following theorem clarifies this relation and gives the necessary and sufficient condition of convergence in the topology~$\tau_J$.
	
	\begin{theorem}[criterion of the transportation convergence]\label{thm:criterion}
		Take measures $\nu^*$ and $\{\nu_n\}_{n \in \nset} \subset \Prsp{X}$. The following conditions are equivalent:
		\begin{enumerate}
			\item $\nu_n \xrightarrow{J} \nu^*$;
			
			\item $\nu_n \rightharpoonup \nu^*$ and $J(\mu, \nu_n) \to J(\mu, \nu^*)$ for any $\mu \in \Prsp{X}$;
			
			\item $\nu_n \rightharpoonup \nu^*$ and $J(\mu, \nu_n) \to J(\mu, \nu^*) < \infty$ for some $\mu \in E(\nu^*)$.
		\end{enumerate}
	\end{theorem}
	
	\begin{proof}
		Clearly, $(1) \Rightarrow (2) \Rightarrow (3)$. Let us show that $(3) \Rightarrow (1)$. Without loss of generality assume that $J(\mu, \nu_n) < \infty$ for all $n$. Let $\gamma_n \in \Pi(\mu, \nu_n)$ be an optimal transport plan. Since the sequence $\{\gamma_n\}_{n \in \nset}$ is tight, one can extract a subsequence $\gamma_n \rightharpoonup \hat\gamma \in \Pi(\mu, \nu^*)$ (without relabelling). Fix $\epsilon > 0$ and an $R > 0$ such that $K'(\hat\gamma) = K'_R(\hat\gamma) > K(\hat\gamma) - \epsilon$. Using the construction above get transport plans $\Pi(\tilde\nu_n, \nu_n) \owns \tilde\gamma_n \rightharpoonup \tilde\gamma\in \Pi(\tilde\nu^*, \nu^*)$. Therefore $J(\tilde\nu_n, \tilde\nu^*) \to 0$, $J(\tilde\nu^*, \nu^*) \le K(\tilde\gamma) = K(\hat\gamma) - K'(\hat\gamma) < \epsilon$ and
		\begin{multline}
			J(\tilde\nu_n, \nu_n) \le K(\tilde\gamma_n) = K(\gamma_n) - K'(\gamma_n) = J(\mu, \nu_n) - K'(\gamma_n)\\
			{}\to J(\mu, \nu^*) - K'(\hat\gamma) \le K(\hat\gamma) - K'(\hat\gamma) < \epsilon.
		\end{multline}
		So, one can construct such sequences $\{\tilde\nu_n\}_{n \in \nset}$, $\{\tilde\nu^*_n\}_{n \in \nset}$ that 
		\[
		\lim J(\tilde\nu_n, \nu_n) = \lim J(\tilde\nu_n, \tilde\nu^*_n) = \lim J(\tilde\nu^*_n, \nu^*) = 0.
		\]
		All the sequences are tight hence $J(\nu_n, \nu^*) \to 0$ by Lemma~\ref{lem:assoc+symmetr}.
	\end{proof}
	
	\begin{remark}
		Obviously, the arguments of $J(\cdot, \cdot)$ can be simultaneously swapped in each of the conditions (2)--(3) without violating the theorem.
	\end{remark}
	
	Moreover, for the class $E_0$ there also exists the following dual formulation of transportation convergence which is extremely similar to the notion of the weak convergence of measures.
	
	\begin{theorem}\label{thm:dual}
		Take measures $\nu^* \in E_0$ and $\{\nu_n\}_{n \in \nset} \subset \Prsp{X}$. Then $\nu_n \xrightarrow{J} \nu^*$ iff $\int f \diff \nu_n \to \int f \diff \nu^*$ for any continuous function $f(\cdot)$ such that $|f(x)| \le \alpha + \beta c(x, x_0)$ for any $x \in X$ and for some constants $\alpha$, $\beta$. 
	\end{theorem}
	
	\begin{proof}
		\begin{enumerate}
			\item Let $\int f \diff \nu_n \to \int f \diff \nu^*$ for any continuous function $f(\cdot)$ such that $|f(x)| \le \alpha + \beta c(x, x_0)$. Then $\nu_n \rightharpoonup \nu^*$ and for any $x_0 \in X$
			\[J(\nu_n, \delta_{x_0}) = \int c(x, x_0) \diff \nu_n \to \int c(x, x_0) \diff \nu^* = J(\nu^*, \delta_{x_0}).\]
			By Theorem~\ref{thm:criterion} that implies $\nu_n \xrightarrow{J} \nu^*$.
			
			\item Now let $\nu_n \xrightarrow{J} \nu^*$. Then $\nu_n \rightharpoonup \nu^*$ and $\int c(x, x_0) \diff \nu_n \to \int c(x, x_0) \diff \nu^*$. Obviously, it is enough to consider the case when $f(\cdot)$ is nonnegative and $f(x) \le 1 + c(x, x_0)$. Consider functions $f_h(\cdot) \eqset \min\{h, f(\cdot)\} \in C_b(X)$, $h \ge 0$. From weak convergence of $\nu_n$ it follows that $\int f_h \diff \nu_n \to \int f_h \diff \nu^*$ for any $h \ge 0$. On the other hand,
			\begin{multline}
				0 \le \int (f - f_h) \diff \nu_n = \int (f - h)_+ \diff \nu_n \le \int \bigl(1 + c(x, x_0) - h\bigr)_+ \diff \nu_n\\
				{} = \int \bigl(c(x, x_0) - c_{h - 1}(x, x_0)\bigr) \diff \nu_n\\
				{} \xrightarrow[n \to \infty]{} \int \bigl(c(x, x_0) - c_{h - 1}(x, x_0)\bigr) \diff \nu^* \to 0 \text{ as } h \to +\infty.
			\end{multline}
			Consequently, $\int f \diff \nu_n \to \int f \diff \nu^*$. 	
		\end{enumerate}
	\end{proof}

	\begin{corollary}\label{cor:polish}
		Let $X$ be a Polish space. Then the space $(E_0, \tau_J)$ is also Polish. 
	\end{corollary}
	
	\begin{proof}
		Consider the following embedding of $E_0$ in the space $\mathcal{M}_+(X)$ of all finite nonnegative Borel measures on $X$: $\nu \mapsto F(\nu) \eqset \bigl(1 + c(x, x_0)\bigr) \lfloor \nu$. From Theorem~\ref{thm:dual} it follows that $\nu_n \xrightarrow{J} \nu^*$ iff $F(\nu_n) \rightharpoonup F(\nu^*)$. Moreover, the image $F(E_0)$ is weakly closed in $\mathcal{M}_+(X)$: indeed, if $\mu_n \eqset F(\nu_n) \rightharpoonup \mu$, then
		\[\int \frac1{1 + c(x, x_0)} \diff \mu = \lim \int \frac1{1 + c(x, x_0)} \diff \mu_n = \int \diff \nu_n = 1,\]
		thus $\nu \eqset \frac1{1 + c(x, x_0)} \lfloor \mu \in E_0$. Space $\mathcal{M}_+(X)$ endowed with the topology of weak convergence is Polish~\cite[Theorem~8.9.3]{Bogachev2007} and $(E_0, \tau_J)$ is isomorphic to its closed subspace, consequently, it also is Polish.
	\end{proof}
	
	As we have seen, the function $\nu \mapsto \int f \diff \nu$ for $f(\cdot)$ from the class $\bigl(1 + c(\cdot, x_0)\bigr) C_b(X)$ is continuous w.r.t.\ the topology $\tau_J$. The next theorem shows that in some cases it is possible to quantify the modulus of continuity of this function.
	
	\begin{theorem}
		Take function $f(\cdot)$ such that $\bigl|f(x) - f(y)\bigr| \le g\bigl(c(x, y)\bigr)$ for all $x, y \in X$ and for some concave function $g(\cdot)$. Then for any $\mu, \nu \in E_0$ the following inequality holds:
		\[
		\Bigl|\int f \diff \mu - \int f \diff \nu\Bigr| \le g\bigl(J(\mu, \nu)\bigr).
		\]
	\end{theorem}
	
	\begin{proof}
		Let $\gamma$ be an optimal transport plan from $\mu$ to $\nu$. Then by Jensen's inequality
		\begin{multline}
			\Bigl|\int f \diff \mu - \int f \diff \nu\Bigr| = \Bigl|\int f(x) \diff \gamma - \int f(y) \diff \gamma\Bigr| \le \int \bigl|f(x) - f(y)\bigr| \diff \gamma\\
			{} \le  \int g\bigl(c(x, y)\bigr) \diff \gamma \le g\Bigl(\int c \diff \gamma\Bigr) = g\bigl(J(\mu, \nu)\bigr).\qedhere
		\end{multline}
	\end{proof}
	
	For example, if $c(x, y) = \rho^p(x, y)$ where $p \ge 1$, and $f(\cdot)$ is Lipschitz continuous with a constant $L$ then $\bigl|\int f \diff \mu - \int f \diff \nu\bigr| \le L J(\mu, \nu)^{1 / p}$. Moreover, in this case $J(\mu, \nu)^{1 / p}$ itself is a metric on \Prsp{X} (see Section~\ref{sec:euclidean}), thus the function $\mu \mapsto \int f \diff \mu$ is also $L$-Lipschitz.
	
	Let us now show that, under Assumption~\ref{asm:general_growth_condition} and local compactness of~$X$, any class $E(\mu_0)$ endowed with the transportation topology is a Radon space. In order to prove this, we show that $\bigl(E(\mu_0), \tau_J\bigr)$ is separable, metrizable, and that any probability measure on this space is tight.
	
	\begin{lemma}\label{lem:separable}
		Take an arbitrary measure $\mu_0 \in \Prsp{X}$. The equivalence class $E(\mu_0)$ endowed with the topology $\tau_J$ is separable and metrizable.
	\end{lemma}
	
	\begin{proof}
		\begin{enumerate}
			\item Let $\rho_w(\cdot, \cdot)$ be a metric on \Prsp{X} that induces weak convergence. Then
			\[\rho_J(\mu, \nu) \eqset \rho_w(\mu, \nu) + \bigl|J(\mu, \mu_0) - J(\nu, \mu_0)\bigr|\]
			is also a metric and, obviously, $\mu_n \xrightarrow{J} \mu^* \in E(\mu_0)$ iff $\rho_J(\mu_n, \mu^*) \to 0$ by Theorem~\ref{thm:criterion}.
	
			\item Let $\mathcal{S}_{\mu_0}$ be a countable family of measures of type $\nu \eqset \bigl(X \setminus B_m(x_0)\bigr)\lfloor \mu_0 + \alpha \sum_{i = 1}^n p_i \delta_{x_i}$ where $m, n\in \nset$, $p_i\in \qset_+$, $x_i$ belong to some countable dense subset of $X$ and $\alpha$ is a normalizing constant. Fix measure $\mu\in C(\mu_0)$, $\epsilon > 0$ and such $R > 0$ that $\int_{B_R(x_0) \times B_R(x_0)} c(x, y) \diff \gamma > K(\gamma) - \epsilon$, where $\gamma$ is an optimal transport plan from $\mu_0$ to $\mu$. Thus $K'(\gamma) \eqset K(f_R\lfloor \gamma) > K(\gamma) - \epsilon$ and $J(\tilde\mu, \mu) \le K(\gamma) - K'(\gamma) < \epsilon$. But $\tilde\mu$ obviously lie in the weak closure of $\mathcal{S}_{\mu_0}$, so there exists a sequence $\mathcal{S}_{\mu_0} \owns \nu_n \xrightarrow{J} \tilde\mu$ and $\lim J(\nu_n, \mu) = J(\tilde\mu, \nu) < \epsilon$. Consequently, $\mathcal{S}_{\mu_0}$ is a dense set in $E(\mu_0)$.
		\end{enumerate}
	\end{proof}
	
	The following assumption about the local compactness of the space~$X$ allows us to obtain the weak local compactness of \Prsp{X}.
	
	\begin{assumption}\label{asm:loc_compact}
		For some (and therefore for any) $x_0 \in X$ any ``closed ball'' $\overline{B}_r^c(x_0) \eqset \bigl\{y \in X: c(x_0, y) \le r\bigr\}$ is compact.
	\end{assumption}
	
	Notice that under this assumption from $c(x, y) = 0$ iff $x = y$ it follows that $c(x, x_n) \to 0$ iff $c(x_n, x) \to 0$ iff $x_n \to x$. Moreover, $X$ is locally-compact and therefore separability of the space implies that it is a Radon space without any additional assumption.
	
	\begin{lemma}\label{lem:tight}
		Let a sequence $\{\nu_n\}_{n \in \nset}$ be such that $\limsup J(\mu, \nu_n) < \infty$ for some $\mu \in \Prsp{X}$; then the sequence is tight.
	\end{lemma}

	\begin{proof}
		Fix $\epsilon > 0$ and a ball $\overline{B}_r^c(x_0)$ such that $\mu\bigl(X \setminus \overline{B}_r^c(x_0)\bigr) < \epsilon / 2$. Consider $R > r$ and any measure $\lambda_R$ such that $\lambda_R\bigl(X \setminus B_R^c(x_0)\bigr) \ge \epsilon$. Then
		\[
		J(\mu, \lambda_R)\ge \frac{\epsilon}2 \inf\bigl\{c(x, y): x\in \overline{B}_r^c(x_0),\, y\notin B_R^c(x_0)\bigr\}\to \infty \text{ as } R \to \infty.
		\]
		But $\limsup J(\mu, \nu_n) < \infty$ and therefore one can find such a compact ball $\overline{B}_{R_\epsilon}^c(x_0)$ that $\nu_n(X \setminus \overline{B}_{R_\epsilon}^c(x_0)) < \epsilon$ for all $n$, i.e. $\{\nu_n\}_{n \in \nset}$ is a tight sequence.
	\end{proof}
	
	\begin{corollary}
		Any ``closed ball'' $\overline{B}_r^J(\mu) \eqset \bigl\{\nu \in \Prsp{X}: J(\mu, \nu) \le r\bigr\}$ is weakly compact.
	\end{corollary}
	
	\begin{proof}
		Consider a sequence $\{\nu_n\}_{n \in \nset} \subset \overline{B}_r^J(\mu)$. It is tight by the Lemma~\ref{lem:tight} and hence there is weakly convergent subsequence $\nu_{n_k} \rightharpoonup \nu^*$. By lemma~\ref{lem:l.s.c.} $J(\mu, \nu^*) \le \liminf J(\mu, \nu_{n_k}) \le r$ thus $\nu^* \in \overline{B}^J_r(\mu)$.
	\end{proof}
	
	Let us show that space \Prsp{X} is ``complete'' with respect to the Monge--Kantorovich distance.
	
	\begin{lemma}\label{lem:complete}
		Let $\{\nu_n\}_{n \in \nset} \subset \Prsp{X}$ be a sequence such that $J(\nu_n, \nu_m) \to 0$ as $n \ge m \to \infty$; then $\nu_n \xrightarrow{J} \nu^*$ for some $\nu^* \in \Prsp{X}$.
	\end{lemma}
	
	\begin{proof}
		Since $\{\nu_n\}_{n \in \nset}$ is a tight sequence by Lemma~\ref{lem:tight} there exists a subsequence $\nu_{n_k} \rightharpoonup \nu^*\in \Prsp{X}$. Assume $\nu_{n_k}$ fails to converge to $\nu^*$; then one can choose such $\epsilon > 0$, $N \in \nset$ that $J(\nu^*, \nu_N) > \epsilon$ and $J(\nu_n, \nu_N) < \epsilon$ for $n \ge N$. But from weak lower semicontinuity of $J(\cdot, \cdot)$ it holds $\epsilon \ge \liminf J(\nu_{n_k}, \nu_N) \ge J(\nu^*, \nu_N) > \epsilon$ what leads to a contradiction. Thereby $\nu_n \xrightarrow{J} \nu^*$.
	\end{proof}
	
	Now it is enough to show that any Borel probability measure over $E(\mu_0)$ is tight in order to prove that $\bigl(E(\mu_0), \tau_J\bigr)$ is a Radon space. Let us adapt to our case the proof for Polish space from~\cite[Theorem~7.1.7]{Bogachev2007}.
	
	\begin{theorem}\label{thm:radon}
		For any $\mu_0 \in \Prsp{X}$ the class $E(\mu_0)$ endowed with the transportation topology $\tau_J$ is a Radon space.
	\end{theorem}
	
	\begin{lemma}[compactness in $\tau_J$]\label{lem:precompactness}
		Let a closed set $\mathcal{H} \subset E(\mu_0)$ be totally bounded with respect to the transportation cost $J(\cdot, \cdot)$, i.e.\ for any $\epsilon > 0$ let there exist measures $\mu_1^\epsilon, \dots, \mu_n^\epsilon$ such that $\mathcal{H} \subset \bigcup_{i = 1}^n B_\epsilon^J(\mu_i^\epsilon)$; then $\mathcal{H}$ is compact in the transportation topology.
	\end{lemma}
	
	\begin{proof}
		Consider an arbitrary sequence $\{\nu_n\}_{n \in \nset} \subset \mathcal{H}$. Using Cantor's diagonal argument and conditions of the lemma one can find a sequence of measures $\{\mu_n\}_{n \in \nset}$ and a subsequence of $\{\nu_n\}_{n \in \nset}$ such that (without relabelling) $J(\mu_n, \nu_k) < 1 / n$ for all $k \ge n$. Then $\{\nu_n\}_{n \in \nset}$ and $\{\mu_n\}_{n \in \nset}$ are tight by Lemma~\ref{lem:tight} and therefore $J(\nu_n, \nu_m) \to 0$ as $m, n \to \infty$ by Lemma~\ref{lem:assoc+symmetr}. Now it follows from Lemma~\ref{lem:complete} that there exists $\nu^* \in \mathcal{H}$ such that $\nu_n \xrightarrow{J} \nu^*$.
	\end{proof}
	
	Now one can prove Theorem~\ref{thm:radon} in the same way as~\cite[Theorem~7.1.7]{Bogachev2007}. The only necessary change regards the criterion of compactness in $\tau_J$, which is considered in Lemma~\ref{lem:precompactness}.

	\section{The case of \texorpdfstring{$\rset^d$}{Rd}}\label{sec:euclidean}
	
	Now consider the locally compact Polish space $X = \rset^d$ with the Euclidean metric and take $c(x, y) = g(x - y)$, where the function $g(\cdot)$ is convex, $g(0) = 0$ and $g(x) > 0$ whenever $x \neq 0$. Obviously, this cost function is continuous and \textit{consistent} in the sense that $c(x, x_n) \to 0$ iff $c(x_n, x) \to 0$ iff $x_n \to x$. Moreover, the space and the cost function satisfy Assumption~\ref{asm:loc_compact}.
	
	Let us assume the following inequality holds:
	\begin{equation}\label{eq:vector_strong_growth_condition}
		B \eqset \sup_{x, y} \frac{g(x + y)}{g(x) + g(y)} < \infty.
	\end{equation}
	Notice that $B \ge 1$ due to convexity of the function $g(\cdot)$.
	
	\begin{theorem}\label{thm:vector_triangle}
		Let inequality~\eqref{eq:vector_strong_growth_condition} hold. Then there exists $q \ge 1$ such that $g^{1/q}(\cdot)$ satisfies the triangle inequality:
		\begin{equation}
			g^{1 / q}(x + y) \le g^{1 / q}(x) + g^{1 / q}(y)
		\end{equation}
		for any $x, y \in \rset^d$
	\end{theorem}
	
	\begin{proof}
		Consider points $x, y \in \rset^d$ such that $g(y) = \xi g(x)$, $\xi \le 1$. Due to convexity of $g(\cdot)$ one can obtain that for any $n \ge 1$ it holds
		\[
		g(x + y) \le \frac{n - 1}n g(x) + \frac1n g(x + n y) \le g(x) + \frac{B}n \bigl(g(x) + g(n y)\bigr).
		\]
		Consider $n = 2^k$; it follows from inequality~\ref{eq:vector_strong_growth_condition} that $g(2^k y) \le (2 B)^k g(y)$ and therefore
		\[
		g(x + y) \le g(x) + 2^{-k} B g(x) + B^{k + 1} g(y) = g(x) \bigl(1 + B (2^{-k} + \xi B^k)\bigr).
		\]
		Take $k \eqset \lfloor -\frac{\ln \xi}{\ln {2 B}}\rfloor$; then 
		\[
		2^{-k} + \xi B^k \le 2^{\frac{\ln \xi}{\ln 2 + \ln B} + 1} + \xi B^{-\frac{\ln \xi}{\ln 2 + \ln B}} = 3 \xi^{1 / q_0},
		\]
		where $q_0 \eqset \frac{\ln {2 B}}{\ln 2} \ge 1$. Thus $g(x + y) \le g(x) \bigl(1 + 3 B \xi^{1 / q_0}\bigr)$. Since $\xi \le 1$ one can obtain that for $q = \max\{3 B, q_0\}$ it holds
		\begin{multline}
			g^{1 / q}(x + y) \le g^{1 / q}(x) \bigl(1 + 3 B \xi^{1 / q_0}\bigr)^{1 / q} \le g^{1 / q}(x) \bigl(1 + \frac{3 B}{q} \xi^{1 / q_0}\bigr)\\
			{} \le g^{1 / q}(x) (1 + \xi^{1 / q}) = g^{1 / q}(x) + g^{1 / q}(y)
		\end{multline}
		for all $x, y \in \rset^d$.
	\end{proof}
	
	\begin{corollary}\label{cor:vector_measure_triangle}
		For all $\mu, \nu, \lambda \in \Prsp{\rset^d}$
		\begin{equation}
			J^{1 / q}(\mu, \nu) \le J^{1 / q}(\mu, \lambda) + J^{1 / q}(\lambda, \nu).
		\end{equation}
	\end{corollary}
	
	\begin{proof}
		Let us take measures $\mu, \nu, \lambda \in \Prsp{\rset^d}$ and optimal transport plans $\gamma_1 \in \Pi(\mu, \lambda)$, $\gamma_2\in \Pi(\lambda, \nu)$. Similarly to the proof of Lemma~\ref{lem:general_weak_triangle} consider a measure $\sigma\in \Pi(\mu, \lambda, \nu)$ such that $\pi^{x, y}_\# \sigma = \gamma_1$, $\pi^{y, z}_\# \sigma = \gamma_2$. Applying Theorem~\ref{thm:vector_triangle} and the Minkowski inequality one can obtain that
		\begin{multline}
			J^{1/q}(\mu, \nu) \le K^{1 / q}(\pi^{x, z}_\# \sigma)\\ 
			{} = \left(\int \bigl(g^{1/q}(x - z)\bigr)^q \diff \sigma\right)^{1 / q} \le \left(\int \bigl(g^{1 / q}(x - y) + g^{1 / q}(y - z)\bigr)^q \diff \sigma\right)^{1 / q}\\
			{} \le \left(\int g(x - y) \diff \sigma\right)^{1 / q} + \left(\int g(y - z) \diff \sigma\right)^{1 / q} = K^{1 / q}(\gamma_1) + K^{1 / q}(\gamma_2)\\
			{} = J^{1 / q}(\mu, \lambda) + J^{1 / q}(\lambda, \nu).\qedhere
		\end{multline}
	\end{proof}
	
	\begin{corollary}\label{cor:metric}
		The function $\rho(\mu, \nu) \eqset \max\{J^{1 / q}(\mu, \nu), J^{1 / q}(\nu, \mu)\} \in [0, \infty]$ is a metric on \Prsp{\rset^d} \textup{(}which may make the value $+\infty$\textup{)}.
	\end{corollary}
	
	As we have seen, under assumption~\eqref{eq:vector_strong_growth_condition} $\bigl(\Prsp{X}, J\bigr)$ is similar to a $q$-Wasserstein space for some degree $q$.
	
	Now consider Assumption~\ref{asm:general_growth_condition} in the Euclidean case. Obviously, one can rewrite it as
	\begin{equation}\label{eq:vector_general_growth_condition}
		g(\pm x \pm y) \le A + B \bigl(g(x) + g(y)\bigr).
	\end{equation}
	
	\begin{theorem}\label{thm:vector_weak_triangle}
		Under Assumption~\ref{asm:general_growth_condition} for any $\epsilon > 0$ there exist $q_\epsilon \ge 1$, $C_\epsilon > 0$ such that for all $x, y \in \rset^d$ it holds
		\begin{align}
			\bigl(g(\pm x \pm y) + \epsilon\bigr)^{1 / q_\epsilon} &\le \bigl(g(x) + \epsilon\bigr)^{1 / q_\epsilon} + \bigl(g(y) + \epsilon\bigr)^{1 / q_\epsilon},\\
			g(x \pm y) &\le \epsilon + (1 + \epsilon) g(x) + C_\epsilon g(y).
		\end{align}
	\end{theorem}
	
	\begin{proof}
		Fix $\epsilon > 0$. Consider $r > 0$ such that $g(x) \le \epsilon$, $\|x\| \le r$. Since $g(x) = 0$ iff $x = 0$ define $a_r \eqset \inf_{\|x\| \ge r} g(x) > 0$. If $\|x + y\| > r$ then $\|x\| > r / 2$ or $\|y\| > r / 2$ therefore
		\begin{multline*}
			g(x + y) \le A + B \bigl(g(x) + g(y)\bigr) \le \frac{A}{a_{r / 2}} \bigl(g(x) + g(y)\bigr) + B \bigl(g(x) + g(y)\bigr) \\
			= \left(\frac{A}{a_{r / 2}} + B\right) \bigl(g(x) + g(y)\bigr).
		\end{multline*}
		Consequently, there exists $D_r = \frac{A}{a_{r / 2}} + B \ge 1$ such that
		\[g(\pm x \pm y) \le \max\{\epsilon, D_r \bigl(g(x) + g(y)\bigr)\}\]
		for all $x, y \in \rset^d$. In particular,
		\[g(\pm x \pm y) + \epsilon \le D_r \bigl(g(x) + \epsilon + g(y) + \epsilon\bigr)\]
		and one can prove the first inequality in the same way as in Theorem~\ref{thm:vector_triangle}.
		
		In order to prove the second inequality, let us choose $k \in \nset$, $r > 0$ such that
		\begin{align}
			& 2^{-k} B < \epsilon,\\
			& 2^{-k} A < \frac{\epsilon}2,\\
			& 2^{-k} B g(x) < \frac{\epsilon}2 \text{ as } \|x\| \le 2^k r.
		\end{align}
		Then similarly to the proof of Theorem~\ref{thm:vector_triangle} one can show that
		\begin{multline}
			g(x + y) \le (1 + 2^{-k} B) g(x) + 2^{-k} B g(2^k y) + 2^{-k} A\\
			{} \le (1 + \epsilon) g(x) + \frac{\epsilon}2 + \max\{\frac{\epsilon}2, D_r^k B g(y)\} \le \epsilon + (1 + \epsilon) g(x) + C_\epsilon g(y),
		\end{multline}
		where $C_\epsilon \eqset D_r^k B$.
	\end{proof}
	
	\begin{corollary}\label{cor:vector_measure_weak_triangle}
		For any $\epsilon > 0$ and measures $\mu, \nu, \lambda \in \Prsp{\rset^d}$ the following inequalities hold:
		\begin{align}
			\bigl(J(\mu, \nu) + \epsilon\bigr)^{1 / q_\epsilon} &\le \bigl(J(\mu, \lambda) + \epsilon\bigr)^{1 / q_\epsilon} + \bigl(J(\lambda, \nu) + \epsilon\bigr)^{1 / q_\epsilon},\\
			J(\mu, \nu) &\le \epsilon + (1 + \epsilon) J(\mu, \lambda) + C_\epsilon J(\lambda, \nu),\\
			J(\mu, \nu) &\le \epsilon + (1 + \epsilon) J(\lambda, \nu) + C_\epsilon J(\mu, \lambda).
		\end{align}
	\end{corollary}
	
	The proof of Corollary~\ref{cor:vector_measure_weak_triangle} is completely similar to the proofs of Lemma~\ref{lem:general_weak_triangle} and Corollary~\ref{cor:vector_measure_triangle}.

	\section{\texorpdfstring{Fr{\'e}chet}{Frechet} barycenters}\label{sec:barycenters}
	
	As we have obtained in sec.~\ref{sec:topology}, the space of probability measures endowed with the transportation topology has some good properties. In this section the \emph{barycenter} of measures will be defined, i.e.\ some kind of averaging w.r.t.\ the transportation structure of the space. It generalizes the construction from~\cite{AguehCarlier2011}, where the 2\nobreakdash{-}Wasserstein space is considered. Other related works on barycenters ate~\cite{LeGouicLoubes2015}, where $p$-Wasserstein barycenters are introduced and studied, and~\cite{KimPass2016} devoted to barycenters of measures on Riemannian manifolds. Regularized barycenters are also considered in the recent work~\cite{BigotCazellesPapadakis2016}.
	
	In the section, the Fr{\'e}chet barycenter will be shown to be ``upper semicontinuous'' in some sense and statistically consistent. Analogous results for measures over~$\rset$ and a convex cost function were proved in~\cite{SobolevskiKroshnin2015}.

	\subsection{Generalized averaging in \texorpdfstring{\Prsp{X}}{P(X)}}\label{subsec:bary}
	
	Let the space \Prsp{X} be endowed with the Borel $\sigma$\nobreakdash{-}algebra $\mathcal{B}(\tau_w)$ induced by the topology of weak convergence $\tau_w$. This $\sigma$-algebra is weaker than $\mathcal{B}(\tau_J)$, induced by the transportation topology. However, as we will see later, they are equivalent for defining an averaging in \Prsp{X}.
	
	\begin{definition}\label{def:empirical_barycenter}
		Take a functional $G \colon \Prsp{X} \to \rset \cup \{+\infty\}$ (regularizer) and consider a finite set $\mu_1, \mu_2, \dots, \mu_n$ of measures in~$\Prsp{X}$ and positive \emph{weights} $\lambda_1 > 0, \lambda_2 > 0, \dots, \lambda_n > 0$.
		A \emph{$G$-regularized Fr{\'e}chet barycenter} $\bary_G(\mu_i, \lambda_i)_{1 \le i \le n} \in \Prsp{X}$ (or just the Fr{\'e}chet barycenter if $G \equiv const$) with respect to the transportation functional $J(\cdot, \cdot)$ is a measure that minimizes
		\begin{equation}
			\sum_{1 \le i \le n} \lambda_i J(\mu_i, \nu) + G(\nu)
		\end{equation}
		over $\nu \in \Prsp{X}$. By $\Bary_G(\mu_i, \lambda_i)_{1 \le i \le n} \in \Prsp{X}$ we denote the set of all Fr{\'e}chet barycenters of $(\mu_i, \lambda_i)_{1 \le i \le n}$.
	\end{definition}
		
	\begin{definition}\label{def:population_barycenter}
		Take a functional $G \colon \Prsp{X} \to \rset \cup \{+\infty\}$ and a distribution $P \in \mathcal{P}\bigl(\Prsp{X}\bigr)$. Consider the problem of minimizing 
		\[
		\E_{\bm\mu \sim P} J(\bm\mu, \nu) + G(\nu) \to \min_{\nu \in \Prsp{X}}
		\]
		and denote any its solution by $\bary_G(P)$. We call the measure $\bary_G(P)$ a \emph{$G$-regularized Fr{\'e}chet barycenter} of the distribution~$P$. Respectively, $\Bary_G(P)$ is the set of all $G$-regularized Fr{\'e}chet barycenters of~$P$.
	\end{definition}
	
	Obviously, Definition~\ref{def:empirical_barycenter} is a particular case of Definition~\ref{def:population_barycenter} hence one can consider distribution $P_n \eqset \sum_{i = 1}^n \lambda_i \delta_{\mu_i}$ such that $\Bary_G(P_n) = \Bary_G(\mu_i, \lambda_i)_{1 \le i \le n}$. Under Assumptions~\ref{asm:general_growth_condition} and~\ref{asm:loc_compact} a $G$-regularized Fr{\'e}chet barycenter with lower semicontinuous and bounded below regularizer~$G$ always exists. 
	
	\begin{theorem}\label{thm:bary_existence}
		Let $G \colon \Prsp{X} \to \rset \cup \{+\infty\}$ be bounded below and lower semicontinuous with respect to the weak convergence, and distribution $P \in \mathcal{P}\bigl(\Prsp{X}\bigr)$ be such that 
		\[
		\inf_{\nu \in \Prsp{X}} \Bigl(\int J(\mu, \nu) \diff P(\mu) + G(\nu)\Bigr) < \infty;
		\] 
		then there exists a $G$-regularized Fr{\'e}chet barycenter of $P$. Moreover, any minimizing sequence $\{\nu_n\}_{n \in \nset}$, i.e.\ such that 
		\[
		\int J(\mu, \nu_n) \diff P(\mu) + G(\nu_n) \to \inf_{\nu \in \Prsp{X}} \Bigl(\int J(\mu, \nu) \diff P(\mu) + G(\nu)\Bigr),
		\] 
		is precompact in the topology $\tau_J$ and every its partial limit is a $G$-regularized barycenter of the distribution. In particular, $\Bary_G(P)$ is compact.
	\end{theorem}
	
	\begin{proof}
		Fix $\epsilon > 0$ and a ball $B = \overline{B}_r^c(x_0)$ such that $\mu(X \setminus B) < \epsilon / 2$ for all measures from some set $\mathcal{U} \subset \Prsp{X}$, $P(\mathcal{U}) > 1 / 2$. Consider any $R > r$ and measure $\lambda_R$ such that $\lambda_R\bigl(X \setminus B_R^c(x_0)\bigr) \ge \epsilon$. One can obtain that
		\[\int J(\mu, \lambda_R) \diff P(\mu) \ge \frac12 \frac{\epsilon}2 \inf\bigl\{c(x, y): x \in B,\, y \notin B_R^c(x_0)\bigr\} \to \infty \text{ as } R \to \infty.\]
		Let $m$ be a lower bound for $G(\cdot)$, then
		\begin{multline}
			\limsup \int J(\mu, \nu_n) \diff P(\mu) \le \limsup \Bigl(\int J(\mu, \nu_n) \diff P(\mu) + G(\nu_n) - m\Bigr)\\
			{} = \inf_{\nu \in \Prsp{X}} \Bigl(\int J(\mu, \nu_n) \diff P(\mu) + G(\nu_n)\Bigr) - m < +\infty,
		\end{multline}
		consequently, $\{\nu_n\}_{n \in \nset}$ is a tight sequence (similarly to Lemma~\ref{lem:tight}) and there exists weakly convergent subsequence $\nu_{n_k} \rightharpoonup \nu^*$.
		
		By Fatou's lemma and lower semicontinuity of $J(\cdot, \cdot)$ and $G(\cdot)$
		\begin{multline}
			\int J(\mu, \nu^*) \diff P(\mu) + G(\nu^*) \le \int \liminf J(\mu, \nu_{n_k}) \diff P(\mu) + \liminf G(\nu_{n_k})\\
			{} \le \liminf \Bigl(\int J(\mu, \nu_{n_k}) \diff P(\mu) + G(\nu_{n_k})\Bigr) = \inf_{\nu \in \Prsp{X}} \Bigl(\int J(\mu, \nu) \diff P(\mu) + G(\nu)\Bigr).
		\end{multline}
		Thus $\nu^* \in \Bary_G(P)$. Moreover, $J(\mu, \nu^*) = \liminf J(\mu, \nu_{n_k})$ for $P$-a.e.\ $\mu$, so
		by Theorem~\ref{thm:criterion} there is a subsequence $\nu_{n_k} \xrightarrow{J} \nu^*$ (without relabelling).
	\end{proof}
	
	In particular, notice that $\bary_G(\mu_i, \lambda_i)_{1 \le i \le n} \in \Prsp{X}$ exists iff all $\mu_i$ belong to the same equivalence class and $G(\cdot) \not\equiv +\infty$ on this class. For $\bary_G(P)$ to exist it is necessary but not sufficient that $\supp P \subset E(\mu)$ for some $\mu$, i.e.\ $P\bigl(E(\mu)\bigr) = 1$. Notice that since $E(\mu)$ and every ball $B_r^J(\nu)$ are measurable w.r.t.\ $\mathcal{B}(\tau_w)$, the restriction of $\mathcal{B}(\tau_w)$ to $E(\mu)$ coincides with $\mathcal{B}(\tau_J)$, as the restriction of $\tau_J$ on $E(\mu)$ has a countable basis of balls. Therefore, it is enough to consider the space \Prsp{X} endowed with $\mathcal{B}(\tau_w)$ instead of a stronger $\sigma$-algebra $\mathcal{B}(\tau_J)$.
	
	As example of bounded below and lower semicontinuous regularizer one can consider a characteristic function of some weakly closed subset $\mathcal{G} \subset \Prsp{X}$, or entropy-type functionals: $G(\nu) = \int g\left(\frac{\diff \nu}{\diff \nu_0}\right) \diff \nu_0$, where $\nu \ll \nu_0$ and $g(\cdot)$ is a convex nonnegative function. See~\cite{BigotCazellesPapadakis2016} for more details on regularized barycenters in the $2$-Wasserstein space.
	
	Consider the case where $G(\cdot)$ is convex. Due to convexity of $J(\cdot, \cdot)$ by Lemma~\ref{lem:convex}, $\Bary_G(P)$ also is a convex set. Moreover, if $X = \rset^d$ and $c(x, y) = g(x - y)$, where $g(\cdot)$ is \emph{strictly} convex, then $J(\mu, \cdot)$ is also strictly convex, whenever $\mu$ is absolutely continuous w.r.t.\ the Lebesgue measure $\mathcal{L}$. It follows from the fact that in this case for any $\nu \sim \mu$ there exists a unique optimal transport plan from $\mu$ to $\nu$ of the form $\gamma = (id \times T)_\# \mu$, where $T$ is an optimal \emph{transport map}~\cite[see][Section~1.3]{Santambrogio2015}. Therefore, there exists a \emph{unique} barycenter of $P$, whenever $P\bigl(\{\mu : \mu \ll \mathcal{L}\}\bigr) > 0$. However, even without any assumption about measures one can take a strictly convex regularizer and ensure the uniqueness of the barycenter.

	\subsection{Consistency of barycenters}\label{subsec:consistency}
	
	Let us consider distributions on $E_0 = E(\delta_{x_0})$ where $x_0$ is some fixed point in $X$. One can define the Monge--Kantorovich distance between them with $J(\cdot, \cdot)$ as a cost function:
	\begin{equation}
		\mathcal{J}(P, P') \eqset \inf_{F \in \Pi(P, P')} \int J(\mu, \nu) \diff F(\mu, \nu),\; P, P' \in \Prsp{E_0}.
	\end{equation}
	As $E_0$ endowed with the topology $\tau_J$ is a Polish space, and $J(\cdot, \cdot)$ as a cost function satisfies Assumption~\ref{asm:general_growth_condition} and other conditions, all the results from section~\ref{sec:topology} hold except for those which need Assumption~\ref{asm:loc_compact}.
	
	Now let us show that convergence of distributions with respect to $\mathcal{J}(\cdot, \cdot)$ implies the transportation convergence of its barycenters. This result is similar to~\cite[Theorem~2]{LeGouicLoubes2015} in case of $p$-Wasserstein spaces. Also, we will obtain the law of large numbers for empirical barycenters proved in~\cite[Theorem~6.1]{BigotKlein2015} for the 2\nobreakdash{-}Wasserstein space and measures with compact support.
	
	First, one can show that for measures from $E_0$ the following analogue of Prokhorov's theorem holds.
	
	\begin{lemma}\label{lem:compactness}
		A family of measures $\mathcal{H} \subset E_0$ is precompact in the transportation topology iff for any $\epsilon > 0$ there exists a compact set $D_\epsilon \subset X$ such that $\int_{X \setminus D_\epsilon} c(x, x_0) \diff \nu \le \epsilon$ for all $\nu \in \mathcal{H}$.
	\end{lemma}
	
	\begin{proof}
		Let the condition of the lemma holds. Consider the image of $\mathcal{H}$ under the isomorphism $F(\nu) \eqset \bigl(1 + c(x, x_0)\bigr) \lfloor \nu$ introduced in the proof of Corollary~\ref{cor:polish}. The set $F(\mathcal{H})$ is tight and uniformly bounded in $\mathcal{M}_+(X)$; indeed, for any $\nu \in \mathcal{H}$ one has
		\begin{multline}
			\int \diff \bigl(F(\nu)\bigr) \eqset \int \bigl(1 + c(x, x_0)\bigr) \diff \nu = 1 + \int_{X \setminus D_1} c(x, x_0) \diff \nu + \int_{D_1} c(x, x_0) \diff \nu\\
			{} \le 2 + \max_{x \in D_1} c(x, x_0) < \infty.
		\end{multline}
		Therefore, $F(\mathcal{H})$ is weakly precompact by the variant of Prokhorov's theorem~\cite[Theorem~8.6.7]{Bogachev2007}. Since $F(E_0)$ is weakly closed in $\mathcal{M}_+(X)$, it follows that $\mathcal{H}$ is precompact in $E_0$ with transportation topology by Theorem~\ref{thm:dual}.
		
		On the other hand, if $\mathcal{H}$ is precompact then $F(\mathcal{H})$ is weakly precompact. Hence $F(\mathcal{H})$ is tight what implies the statement of the lemma.
	\end{proof}
	
	\begin{theorem}\label{thm:bary_consistency}
		Take a regularizer $G$ which is bounded below and lower semocontinuous with respect to $\tau_w$. Let a sequence $\{P_n\}_{n \in \nset} \subset \Prsp{E_0}$ be such that $P_n \xrightarrow{\mathcal{J}} P$ for some distribution $P$ and let there exist a $G$-regularized Fr{\'e}chet barycenter of~$P$. Then there exist barycenters $\nu_n \in \Bary_G(P_n)$ beginning from some $n_0$, the sequence $\{\nu_n\}_{n \ge n_0}$ is precompact and every its partial limit is a barycenter of~$P$. In particular, if $\nu^* \eqset \bary_G(P)$ is unique, then $\nu_n \xrightarrow{J} \nu^*$.
	\end{theorem}
	
	\begin{remark}
		One can rewrite the statement of the theorem as follows: let the distribution $P$ have a $G$-regularized Fr{\'e}chet barycenter; then for any $\epsilon > 0$ there exists $\delta > 0$ such that 
		\[\Bary_G(P') \subset U_\epsilon\bigl(\Bary_G(P)\bigr) \text{ as } \mathcal{J}(P, P') < \delta,\]
		where $U_\delta\bigl(\Bary_G(P)\bigr) \eqset \bigcup_{\mu \in \Bary_G(P)} B^J_\epsilon(\mu)$ is an open neighbourhood of $\Bary_G(P)$. One can say that set-valued map $P \mapsto \Bary_G(P)$ is upper-semicontinuous with respect to a Hausdorf-like distance.
	\end{remark}
	
	\begin{proof}
		Obviously, for any $\mu_0 \in E_0$ it holds $\limsup J(\nu_n, \mu_0) < \infty$ therefore $\{\nu_n\}_{n \in \nset}$ is tight. Let $\nu_n \rightharpoonup \nu^*$ without relabelling. Assume that there is no subsequence convergent to $\nu^*$ in $\tau_J$. Then by Lemma~\ref{lem:compactness} one can assume without loss of generality that for some $\epsilon_0 > 0$ and any $R > 0$
		\[
		\liminf \int_{X \setminus B_R} c(x, x_0) \diff \nu_n > \epsilon_0 B.
		\]
		For any $R > 0$ fix continuous function $f_R \colon X \to [0, 1]$ such that $f_R(x) = 1$, $x \in B_R$ and $f_R(x) = 0$, $x \notin B_{R + 1}$. For given $R$ and measure $\nu$ let us define measure $\tilde\nu  = \tilde\nu_R \eqset f_R \lfloor \nu + \bigl(1 - (f_R \lfloor \nu)(X)\bigr) \delta_{x_0}$ similarly to one in Section~\ref{sec:topology}. Notice that $\tilde\nu_n \xrightarrow{J} \tilde\nu^*$ for any $R$ and $\tilde\nu^* \xrightarrow{J} \nu^*$ as $R \to \infty$.
		
		Consider an arbitrary measure $\mu \in E_0$. Due to Lemma~\ref{lem:compactness},
		\[
		\sup_{\mu' \in B^J_\delta(\mu)} \int_{X \setminus B_r} c(x, x_0) \diff \mu' \to 0, \text{ as } r \to \infty,\; \delta \to 0.
		\]
		Now one can choose $r, R, \delta > 0$ and $n_0 \in \nset$ such that for any $n \ge n_0$ and for any $\mu' \in B^J_\delta(\mu)$ the following inequalities hold:
		\begin{align}
			& J(\mu', \nu^*) < J(\mu', \tilde\nu_n) + \frac{\epsilon_0}{4},\\
			& \int_{X \setminus B_{R + 1}} c(y, x_0) \diff \nu_n(y) > \epsilon_0 B,\\
			& \int_{X \setminus B_r} c(x, x_0) \diff \mu'(x) < \frac{\epsilon_0}{8},\\
			& r \nu_n\bigl(X \setminus B_R\bigr) < \frac{\epsilon_0}{8},\\
			& \nu_n\bigl(X \setminus B_{R + 1}\bigr) < \epsilon_0 \frac{B}{8 A}.
		\end{align}
		
		Let us denote by $\gamma_{\mu'}^{\nu_n}$ an optimal transport plan from $\mu'$ to $\nu_n$. Then 
		\[
		J(\mu', \tilde\nu_n) \le \int_{X \times B_{R + 1}} c \diff \gamma_{\mu'}^{\nu_n} + \int_{X \times (X \setminus B_R)} c(x, x_0) \diff \gamma_{\mu'}^{\nu_n}.
		\]
		One can obtain that 
		\begin{multline}
			\int_{X \times (X \setminus B_R)} c(x, x_0) \diff \gamma_{\mu'}^{\nu_n}
			{} \le \int_{B_r \times (X \setminus B_R)} c(x, x_0) \diff \gamma_{\mu'}^{\nu_n} + \int_{(X \setminus B_r) \times X} c(x, x_0) \diff \gamma_{\mu'}^{\nu_n}\\
			{} \le r \nu_n(X \setminus B_R) + \int_{X \setminus B_r} c(x, x_0) \diff \mu' \le \frac{\epsilon_0}{4},
		\end{multline}
		hence
		\[
		J(\mu', \tilde\nu_n) \le \int_{X \times B_{R + 1}} c(x, y) \diff \gamma_{\mu'}^{\nu_n} + \frac{\epsilon_0}{4}.
		\]
		On the other hand,
		\begin{multline}
			J(\mu', \nu_n) = \int_{X \times B_{R + 1}} c(x, y) \diff \gamma_{\mu'}^{\nu_n} + \int_{X \times (X \setminus B_{R + 1})} c(x, y) \diff \gamma_{\mu'}^{\nu_n}\\
			{} \ge \int_{X \times B_{R + 1}} c(x, x_0) \diff \gamma_{\mu'}^{\nu_n} + \int_{X \times (X \setminus B_{R + 1})} \bigl(\frac1B c(y, x_0) - c(x, x_0) - \frac{A}{B}\bigr) \diff \gamma_{\mu'}^{\nu_n}\\
			{} = \int_{X \times B_{R + 1}} c(x, x_0) \diff \gamma_{\mu'}^{\nu_n} + \frac1B \int_{X \setminus B_{R + 1}} c(y, x_0) \diff \nu_n\\
			{} - \int_{X \times (X \setminus B_{R + 1})} c(x, x_0) \diff \gamma_{\mu'}^{\nu_n} - \frac{A}{B} \nu_n(X \setminus B_{R + 1})\\
			{} \ge \int_{X \times B_{R + 1}} c(x, x_0) \diff \gamma_{\mu'}^{\nu_n} + \epsilon_0 - \frac{\epsilon_0}{4} - \frac{\epsilon_0}{8}\\
			{} = \int_{X \times B_{R + 1}} c(x, x_0) \diff \gamma_{\mu'}^{\nu_n} + \frac{5 \epsilon_0}{8}.
		\end{multline}
		Therefore $J(\mu', \tilde\nu_n) \le J(\mu', \nu_n) - 3 \epsilon_0 / 8$ and
		\[
		J(\mu', \nu^*) < J(\mu', \tilde\nu_n) + \frac{\epsilon_0}{4} \le J(\mu', \nu_n) - \frac{\epsilon_0}{8}
		\]
		for all $\mu' \in B^J_\delta(\mu)$. Thus for any $k \in \nset$ there exists an open set $U_k \subset B^J_k(\nu^*)$ such that for all $\mu \in U_k$ 
		\[
		J(\mu, \nu_n) \ge J(\mu, \nu^*) + \frac{\epsilon_0}{8},\; n \ge k,
		\]
		and $\bigcup_{k = 1}^\infty U_k = E_0$. Now one can obtain that for any $k$
		\begin{align}
			\liminf \Bigl(\int J(\mu, \nu_n) \diff P_n + G(\nu_n)\Bigr) &\ge \liminf \Bigl(\int_{U_k} J(\mu, \nu_n) \diff P_n + G(\nu_n)\Bigr)\\
			{} &\ge \liminf \Bigl(\int_{U_k} J(\mu, \nu^*) \diff P_n + \frac{\epsilon_0}{8} P_n(U_k) + G(\nu_n)\Bigr) \\
			{} &\ge \int_{U_k} J(\mu, \nu^*) \diff P + \frac{\epsilon_0}{8} P(U_k) + G(\nu^*).
		\end{align}
		But since $\bigcup_{k = 1}^\infty U_k = E_0$
		\[
		\liminf \Bigl(\int J(\mu, \nu_n) \diff P_n + G(\nu_n)\Bigr) \ge \int J(\mu, \nu^*) \diff P + G(\nu^*) + \frac{\epsilon_0}{8},
		\]
		what contradicts to the fact that
		\[
		\int J(\mu, \nu_n) \diff P_n + G(\nu_n) \le \int J(\mu, \nu^*) \diff P_n + G(\nu^*) \to \int J(\mu, \nu^*) \diff P + G(\nu^*).
		\]
		Consequently, there exists a convergent subsequence of barycenters $\nu_{n_k} \xrightarrow{J} \nu^*$. 
		
		Let us show that $\nu^*$ is a $G$-regularized barycenter of~$P$. Indeed, consider any $\nu \in \Prsp{X}$. Continuity of the transportation distance implies 
		\begin{multline}
			\int J(\mu, \nu^*) \diff P + G(\nu^*) \le \lim \int J(\mu, \nu_{n_k}) \diff P_{n_k} + \liminf G(\nu_{n_k})\\
			{} \le \liminf \Bigl(\int J(\mu, \nu) \diff P_{n_k} + G(\nu)\Bigr) = \int J(\mu, \nu) \diff P + G(\nu),
		\end{multline}
		thus, $\nu^* \in \Bary_G(P)$.
	\end{proof}

	\begin{remark}
		Although the set-valued map $P \mapsto \Bary_G(P)$ is in some sense ``upper semicontinuous'', in general case there do not exist a \emph{continuous} function $P \mapsto \bary_G(P)$, even for $G(\cdot) \equiv const$. However, if $G(\cdot)$ is strictly convex then $P \mapsto \bary_G(P)$ is actually continuous.
	\end{remark}
	
	\begin{corollary}[upper semicontinuity of empirical barycenters]\label{cor:bary_continuity}
		Take a sequences of measures $\{\mu_i^n\}_{n \in \nset} \subset \Prsp{X}$ and weights $\{\alpha_i^n\}_{n \in \nset} \subset [0, +\infty)$ such that $\mu_i^n \xrightarrow{J} \mu_i$, $\alpha_i^n \to \alpha_i$ for $1 \le i \le m$. Then, if all $\mu_i \in E_0$, the sequence of Fr{\'e}chet barycenters $\{\bary_G(\mu_i^n, \alpha_i^n)_{1 \le i \le m}\}_{n \in \nset}$ is precompact and every its partial limit belongs to $\Bary_G(\mu_i, \alpha_i)_{1 \le i \le m}$.
	\end{corollary}
	
	\begin{proof}
		Let us consider $P_n \eqset \sum_i \alpha_i^n \delta_{\mu_i^n}$, $P \eqset \sum_i \alpha_i \delta_{\mu_i}$, so $\Bary_G(\mu_i^n, \alpha_i^n)_{1 \le i \le m} = \Bary_G(P_n)$ and $\Bary_G(\mu_i, \alpha_i)_{1 \le i \le m} = \Bary_G(P)$.  Notice that $J(\mu_i^n, \mu_j) \to J(\mu_i, \mu_j)$ and $\max_{i, j} J(\mu_i, \mu_j) < \infty$ hence
		\[
		\mathcal{J}(P_n, P) \le \sum_{i = 1}^m \min\{\alpha_i^n, \alpha_i\} J(\mu_i^n, \mu_i) + \max_{i, j} J(\mu_i^n, \mu_j) \sum_{i = 1}^m |\alpha_i^n - \alpha_i| \to 0.
		\]
		This shows that the conditions of Theorem~\ref{thm:bary_consistency} hold.
	\end{proof}
	
	\begin{corollary}[law of large numbers]\label{cor:empirical_bary_consistency}
		Let $\left\{\bm\mu_n\right\}_{n \in \mathbb{N}} \subset E_0$ be a sequence of i.i.d.\ random elements with distribution $P_\mu$ such that there exists $\bary_G(P_\mu)$, and $\overline{\bm\mu}_n \in \Bary_G(\bm\mu_i, 1 / n)_{1 \le i \le n}$ be a measurable choice of empirical Fr{\'e}chet barycenters. Then the sequence $\{\overline{\bm\mu}_n\}_{n \in \nset}$ is precompact a.s.\ and every its partial limit is a barycenter of the distribution $P_\mu$.
	\end{corollary}
	
	\begin{proof}
		Let us consider empirical measures $\bm{P}_n \eqset \frac1n \sum_{i = 1}^n \delta_{\bm\mu_i}$. Obviously, the empirical barycenter $\bm\nu_n \eqset \bary_G(\bm\mu_i, 1 / n)_{1 \le i \le n} = \bary_G(\bm{P}_n)$. By the law of large numbers
		\[
		\mathcal{J}(\bm{P}_n, \delta_{\nu^*}) = \frac1n \sum_i J(\bm\mu_i, \nu^*) \to \E J(\bm\mu, \nu^*) = \mathcal{J}(P_\mu, \nu^*) < \infty \text{ a.s.},
		\]
		and $\bm{P}_n \rightharpoonup P_\mu$ since the topology of the weak convergence in $\Prsp{E_0}$ has a countable basis due to separability of $E_0$. Then by Theorem~\ref{thm:criterion} $\bm{P}_n \xrightarrow{\mathcal{J}} P_\mu$ almost surely, i.e.\ the conditions of the theorem hold.
	\end{proof}
	
	Actually, the results just proved also hold for an arbitrary equivalence class under some additional assumption: let for any $\epsilon > 0$ there exist constants $A_\epsilon, C_\epsilon > 0$ such that
	\begin{align}\label{eq:strong_growth_condition}
		c(x, y) &\le A_\epsilon + (1 + \epsilon) c(x, z) + C_\epsilon c(y, z),\\
		c(x, y) &\le A_\epsilon + (1 + \epsilon) c(z, y) + C_\epsilon c(z, x),
	\end{align}
	for all $x, y, z \in X$. This condition is stronger than that in Assumption~\ref{asm:general_growth_condition}, but they coincide i.e.\ in Euclidean case with convex cost function considered in section~\ref{sec:euclidean}  (Corollary~\ref{cor:vector_measure_weak_triangle}). Under such an assumption Theorem~\ref{thm:bary_consistency} holds for any equivalence class  $E(\mu_0)$.
	
	Notice that all the statements in this section also hold for the space \Prsp{X} instead of $\mathcal{P}\bigl(\Prsp{X}\bigr)$ because one can identify a point $x \in X$ with a Dirac measure $\delta_x \in \Prsp{X}$ so that $J(\delta_x, \delta_y) = c(x, y)$ for all $x, y \in X$.


\begin{thebibliography}{10}
		\providecommand{\natexlab}[1]{#1}
		\providecommand{\url}[1]{\texttt{#1}}
		\expandafter\ifx\csname urlstyle\endcsname\relax
		  \providecommand{\doi}[1]{doi: #1}\else
		  \providecommand{\doi}{doi: \begingroup \urlstyle{rm}\Url}\fi
		
		\bibitem[Agueh and G.(2011)]{AguehCarlier2011}
		M.~Agueh and Carlier. G.
		\newblock {Barycenters in the Wasserstein Space}.
		\newblock \emph{SIAM Journal on Mathematical Analysis}, 43\penalty0
		  (2):\penalty0 904--924, 2011.
		
		\bibitem[Ambrosio et~al.(2008)Ambrosio, Gigli, and
		  Savar{\'e}]{AmbrosioGigliSavare2008}
		L.~Ambrosio, N.~Gigli, and G.~Savar{\'e}.
		\newblock \emph{Gradient Flows in Metric Spaces and in the Space of Probability
		  Measures}.
		\newblock Birkh{\"a}user, Basel, 2008.
		
		\bibitem[Santambrogio(2015)]{Santambrogio2015}
		F.~Santambrogio.
		\newblock \emph{Optimal Transport for Applied Mathematicians}.
		\newblock Birkh{\"a}user, Basel, 2015.
		
		\bibitem[Villani(2009)]{Villani2009}
		C.~Villani.
		\newblock \emph{Optimal Transport, Old and New}.
		\newblock Springer-Verlag, Berlin--Heidelberg, 2009.
		
		\bibitem[Bogachev(2007)]{Bogachev2007}
		V.~Bogachev.
		\newblock \emph{Measure Theory}, volume~2.
		\newblock Springer-Verlag, Berlin--Heidelberg, 2007.
		
		\bibitem[Gouic and Loubes(2015)]{LeGouicLoubes2015}
		T.~Le Gouic and J.-M. Loubes.
		\newblock {Existence and Consistency of Wasserstein Barycenters}.
		\newblock \emph{ArXiv e-prints}, 2015.
		\newblock URL \url{http://arxiv.org/abs/1506.04153v1}.
		
		\bibitem[Kim and Pass(2014)]{KimPass2016}
		Y.-H. Kim and B.~Pass.
		\newblock {Wasserstein Barycenters over Riemannian manifolds}.
		\newblock \emph{ArXiv e-prints}, 2014.
		\newblock URL \url{http://arxiv.org/abs/1412.7726}.
		
		\bibitem[Bigot et~al.(2016)Bigot, Cazelles, and
		  Papadakis]{BigotCazellesPapadakis2016}
		J.~Bigot, E.~Cazelles, and N.~Papadakis.
		\newblock {Regularization of barycenters in the Wasserstein space}.
		\newblock \emph{ArXiv e-prints}, 2016.
		\newblock URL \url{http://arxiv.org/abs/1606.01025v1}.
		
		\bibitem[Kroshnin and Sobolevski(2015)]{SobolevskiKroshnin2015}
		A.~Kroshnin and A.~Sobolevski.
		\newblock {Fr{\'e}chet Barycenters and a Law of Large Numbers for Measures on
		  the Real Line}.
		\newblock \emph{ArXiv e-prints}, 2015.
		\newblock URL \url{http://arxiv.org/abs/1512.08421}.
		
		\bibitem[Bigot and Klein(2015)]{BigotKlein2015}
		J.~Bigot and T.~Klein.
		\newblock {Consistent Estimation of a Population Barycenter in the Wasserstein
		  Space}.
		\newblock \emph{ArXiv e-prints}, 2015.
		\newblock URL \url{http://arxiv.org/abs/1212.2562v4}.	
	\end{thebibliography}

\end{document}